\documentclass[10pt,letterpaper]{amsart}
\usepackage{graphicx,hhline}
\usepackage{amsfonts,amsthm,amsmath,amssymb,latexsym, amscd, euscript}
\usepackage{graphicx,color}
\usepackage[all]{xy}
\usepackage{epsfig}
\usepackage{soul}
\usepackage{float}
\usepackage{mathrsfs}
\usepackage{hyperref}
\hypersetup{colorlinks=true,linkcolor=blue,citecolor=red}
\usepackage{amssymb}
\usepackage{epstopdf}

\theoremstyle{plain}
\newtheorem{theorem}{Theorem}[section]
\newtheorem{corollary}[theorem]{Corollary}

\newtheorem{definition}[theorem]{Definition}
\newtheorem{conjecture}[theorem]{Conjecture}

\newtheorem{problem}[theorem]{Problem}

\begin{document}

	\theoremstyle{definition} 

	\newtheorem*{notation}{Notation}  

	\theoremstyle{plain}      

	\def\H{{\mathbb H}}
	\def\F{{\mathcal F}}
	\def\R{{\mathbb R}}
	\def\Q{\hat{\mathbb Q}}
	\def\Z{{\mathbb Z}}
	\def\E{{\mathcal E}}
	\def\N{{\mathbb N}}
	\def\X{{\mathcal X}}
	\def\Y{{\mathcal Y}}
	\def\C{{\mathbb C}}
	\def\D{{\mathbb D}}
	\def\G{{\mathcal G}}
	\def\T{{\mathcal T}}

	\title[Recurrent flows by twisting]{Inducing recurrent flows by twisting on infinite surfaces with unbounded cuffs}
	
	
	\date{}
	
	\author{Hrant Hakobyan, Michael Pandazis and Dragomir \v Sari\'c}
	
	\address[Hrant Hakobyan]{Department of Mathematics, Kansas State University\\ Manhattan, KS, 66506-2602, USA.}
	\email{hakobyan@math.ksu.edu}
	
	\address[Michael Pandazis]{Department of Mathematics and Computer Science, John Jay, CUNY \\ 524 West 59th Street,
N.Y., N.Y., 10019, USA.}
	\email{{mpandazis@jjay.cuny.edu}}

	\address[Dragomir \v Sari\' c]{PhD Program in Mathematics, The Graduate Center, CUNY \\ 365 Fifth Ave., N.Y., N.Y., 10016 and\newline Department of Mathematics, Queens College, CUNY\\ 65--30 Kissena Blvd., Flushing, NY 11367, USA.}
	\email{Dragomir.Saric@qc.cuny.edu}
	
	\thanks{The first author was partially supported by Simons Foundation Collaboration Grant, award ID: 638572.}
	\thanks{The third author was partially supported by PSC-CUNY grants.}

	\begin{abstract}
		
		A Riemann surface $X$ is parabolic if and only if the geodesic flow (for the hyperbolic metric) on the unit tangent bundle of $X$ is ergodic. Consider a Riemann surface $X$ with a single topological end and a sequence $\alpha_n$ of pairwise disjoint, simple closed geodesics converging to the end, called {\it cuffs}. Basmajian, the first and the third author, proved that when the lengths $\ell (\alpha_n)$ of cuffs are at most $2\log n$, the surface $X$ is parabolic. 
		One could expect that having arbitrary large cuff lengths $\ell (\alpha_n)$ (think of $\ell (\alpha_n)=n!^{n!}$)
		would allow the geodesic flow to escape to infinity, thus making $X$ not parabolic. 
		
		Contrary to this and motivated by their proof of the Surface Subgroup Theorem,
		Kahn and Markovi\'c conjectured that for every choice of lengths $\ell (\alpha_n)$, there is a choice of twists that would make $X$ parabolic. We show that their conjecture is essentially true. Namely, for any sequence of positive numbers $\{ a_n\}$, there is a choice of lengths $\ell (\alpha_n)\geq a_n$ such that the (relative) twists by $1/2$ make $X$ parabolic. This result extends to the surfaces with countably many ends while it does not hold for uncountably many ends.
	\end{abstract}
	
	\maketitle
	
	\section{Introduction}
	

	A Riemann surface $X$ is {\it parabolic} if it does not support a Green's function, which we denote by $X\in O_G$. 
	Then, $X\in O_G$ if and only if the geodesic flow on the unit tangent bundle $T^1X$ is ergodic if and only if the Brownian motion on $X$ recurrent if and only if almost every geodesic on $X$ is recurrent, see \cite{Hopf,Nicholls,Sullivan,Tsuji}. For an additional (yet partial) list of equivalent conditions to $X\in O_G$ see \cite[Introduction]{BHS} as well as \cite{Bishop,Nevanlinna:criterion,Nicholls1,Fernandez-Melian,Astala-Zinsmeister,Sullivan}.
	
	In this paper we study the \emph{type {problem} for infinite Riemann surfaces}, that is we would like to determine whether a Riemann surface obtained by an explicit construction is parabolic or not. For us the Riemann surface $X$ is equipped with the hyperbolic metric, that is $X=\mathbb{H}/\Gamma$ where $\Gamma$ is a Fuchsian group. We say that $X$ is {\it infinite} if $\Gamma$ is not finitely generated. We assume that $\Gamma$ is of the first kind, i.e. the limit set $\Lambda(\Gamma)=\mathbb{S}^1$,  since otherwise $X$ is easily seen to be not parabolic. 
	
	It is known that $X$ can be obtained by gluing countably many geodesic pairs of pants via isometries along their cuffs (i.e. boundary geodesics of pairs of pants), see \cite{AlvarezRodriguez} and also \cite{Basmajian,BasmajianSaric,Portilla}. The geodesic pairs of pants are uniquely determined by the cuff lengths, and the isometric gluings along two cuffs are determined by a real parameter called the {\it twist}. 
	The \textit{Fenchel-Nielsen parameters} of $X$ are the pair of sequences $(\{ \ell_n\},\{t_n\})$, where $\ell_n$ is the length and $t_n$ is the twist of the $n$-th cuff, see Figure \ref{fig:flute} for an example. Therefore, the set of all hyperbolic metrics on a topological surface $X$ (modulo isometries and markings) corresponds to all choices of the Fenchel-Nielsen parameters.

	The type problem for Riemann surfaces has been studied extensively by many authors when the Riemann surfaces were naturally defined by either gluing construction along the slits or by covering maps or other constructions motivated by complex analysis considerations, see e.g.  \cite{AhlforsSario,Doyle,LS,Milnor,Mer}. In  \cite{BHS} the question of deciding when $X\in O_G$ from the data of its Fenchel-Nielsen parameters $(\{ \ell_n\},\{t_n\})$ was studied.  A particularly interesting class of examples considered in \cite{BHS} are the Riemann surfaces with one infinite topological end: the (tight) flute surfaces and Loch-Ness monsters.

	\subsection{Tight flute surfaces}
	
	A Riemann surface $X$ is said to be a \textit{(tight) flute surface}, if it has countably many punctures that accumulate to a single non-isolated topological end. We will fix the pants decomposition of $X$ with cuffs $\{\alpha_n\}_{n=1}^{\infty}$ as in Figure \ref{fig:flute}. Let $\ell_n$ and $t_n$ be the length and twist parameters of $\alpha_n$. The twist parameter $t_n$, for $-1/2\leq t_n\leq 1/2$, represents the (oriented) relative length of the arc between the feet of $\gamma_n'$ and $\gamma_{n-1}''$ along $\alpha_n$, see Figure \ref{fig:flute}. 
	By \cite[Theorem 9.1]{BHS}, a flute (Riemann) surface $X(\{\ell_n\},\{t_n\})$ is parabolic if 
	\begin{equation*}
		\label{eq:flute-suff-O_G}
		\ell_n\leq 2\log n
	\end{equation*}
	for all but finitely many $n$, for any choise of the twists $\{t_n\}$.

	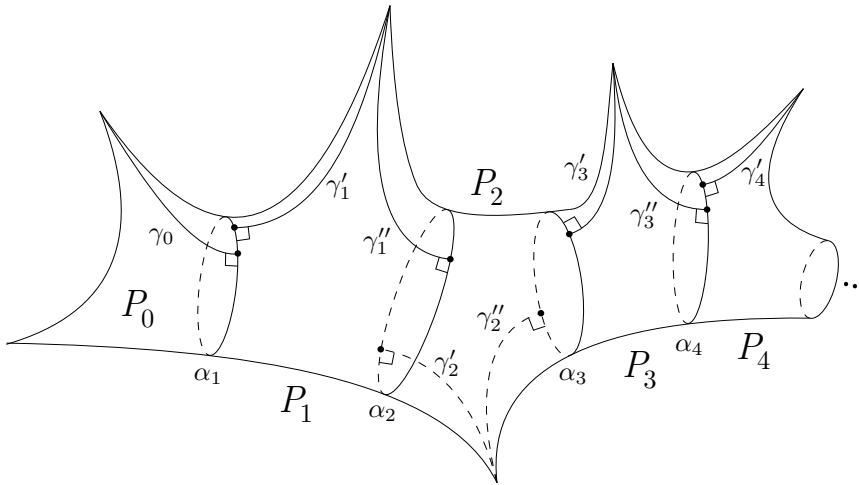
\begin{figure}[htb]
		\begin{picture}(350,180)
			\put(20,0)
			{\includegraphics[width = 4.5in, height = 2.5in]{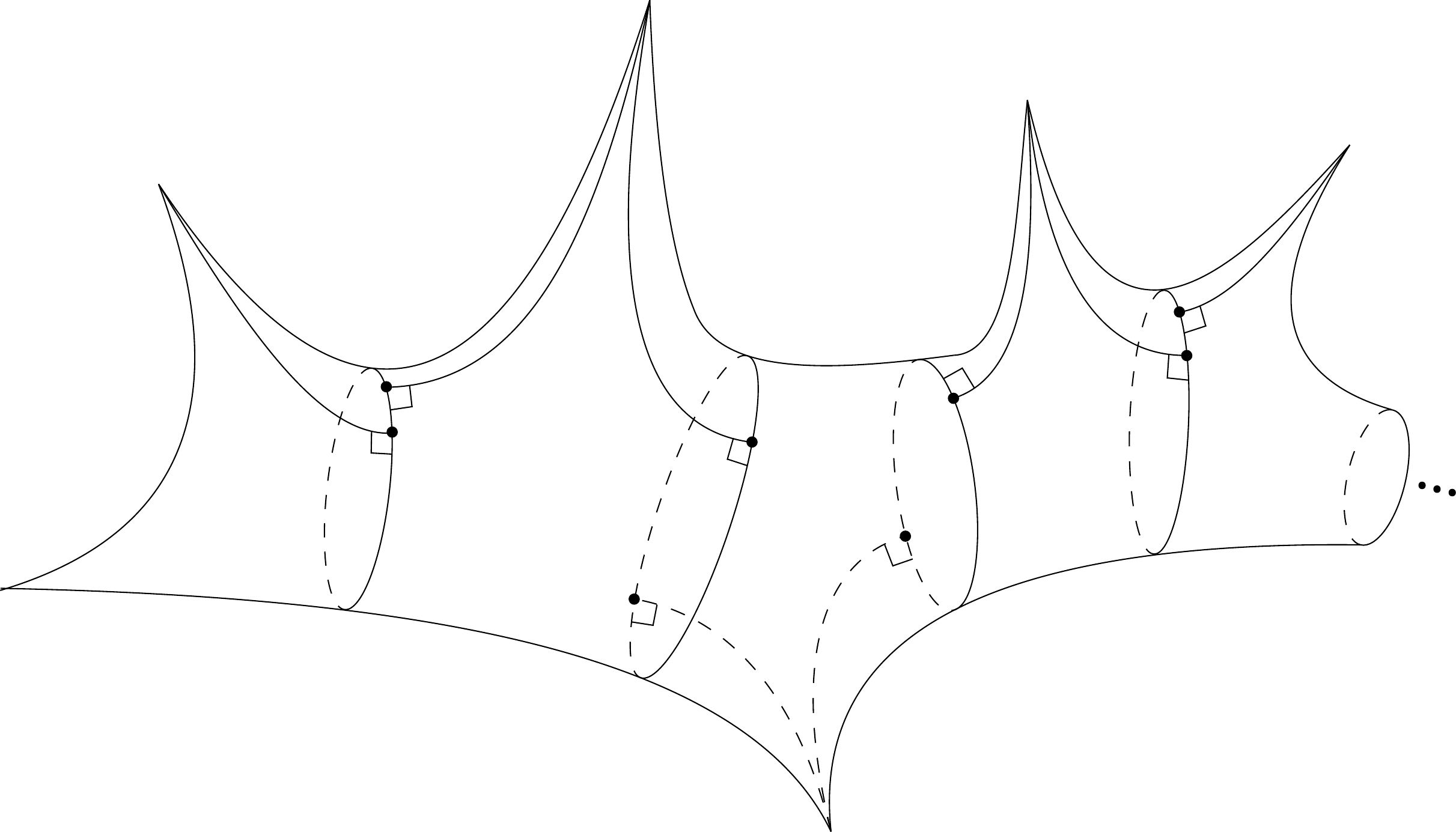}}
			\put(63,63){\LARGE \it P}
			\put(71,59){0}
			\put(123,27){\LARGE \it P}
			\put(131,23){1}
			\put(195,109){\LARGE \it P}
			\put(203,105){2}
			\put(253,40){\LARGE \it P}
			\put(261,36){3}
			\put(296,47){\LARGE \it P}
			\put(304,43){4}
			\put(91,39){$\alpha_1$}
			\put(157,25){$\alpha_2$}
			\put(229,39){$\alpha_3$}
			\put(273,50){$\alpha_4$}
			\put(74,92){$\gamma_0$}
			\put(141,110){$\gamma_1'$}
			\put(155,89){$\gamma_1''$}
			\put(182,43){$\gamma_2'$}
			\put(198,60){$\gamma_2''$}
			\put(231,117){$\gamma_3'$}
			\put(256,99){$\gamma_3''$}
			\put(298,115){$\gamma_4'$}
		\end{picture}
		\caption{A flute surface}
		\label{fig:flute}
	\end{figure}

	Intuitively, the geodesic flow on the unit tangent bundle $T^1X$ of $X$ {\it is not} ergodic if there is a ``lot of space'' on $X$ for the geodesics to escape towards the end space $\partial_{\infty}X$. A lot of space could mean that the size of the ``openings'' $\ell_n$ converges to $\infty$ very fast.  
	To support this intuitive reasoning, consider the {\it zero-twist} flute surface $X=X(\{\ell_n\},\{0\})$, i.e. $t_n=0$ for all $n$. By Theorem 9.4 of \cite{BHS},
	$X(\{\ell_n\},\{0\})$ is not parabolic if,   $$\ell_n\geq p\log n,$$ 
	for a fixed $p>2$ and for all but finitely many $n$.
	
	The {\it half-twist} flute surface $X(\{\ell_n\},\{t_n\})$ is defined by the condition $t_n=1/2$ for all $n$. Basmajian, the first and the third author \cite[Theorem 9.7]{BHS} established that a half-twist flute surface {$X(\{\ell_n\},\{1/2\})$} with {\it increasing} and {\it concave} lengths of cuffs is not parabolic if (and only if) $\ell_n\geq p\log n$ {for all but finitely many $n$ and} for a fixed $p>4$. This further supports the intuitive reasoning of large cuffs preventing the geodesic flow from being ergodic. On the other hand, the second and the third author \cite{PandazisSaric} showed that the condition that the lengths are concave {\it cannot be removed}. Namely, \cite{PandazisSaric} finds a class of half-twist flutes that are parabolic with $$q\log n\geq\ell_n\geq p\log n$$ for any $q>p>0$. 
	
	In two separate discussions with Jeremy Kahn and Vladimir Markovi\'c, the following conjecture was made,  see also Question 1.9(3) in \cite{BHS}.
	
	\begin{conjecture}[Kahn-Markovi\' c]\label{conj}
		Given a sequence $\{ \ell_n\}$ of non-decreasing positive numbers (possibly $\lim_{n\to\infty}\ell_n=\infty$), there always exists a choice of twists $\{ t_n\}$ such that the geodesic flow on the unit tangent bundle of the flute surface $X(\{\ell_n\},\{t_n\})$ is ergodic, i.e. $X(\{ \ell_n\},\{t_n\})\in O_G$.
	\end{conjecture}
	
	The conjecture proposes that the influence of the twists is strong enough to make a flute surface parabolic,  even if the lengths of cuffs increase arbitrarily fast. This conjecture was motivated by the proof of the Surface Subgroup Theorem \cite{KM}. Prior to the conjecture, Basmajian and the third author \cite{BasmajianSaric} showed that for any sequence of lengths $\{ \ell_n\}$ there is a choice of twists $\{ t_n\}$ such that $X(\{\ell_n\},\{t_n\})$ has covering group of the first kind (which is a necessary but not sufficient condition for parabolicity). 
	
	In this paper, we prove Conjecture \ref{conj} under the additional assumption that the cuff lengths can be taken larger than the assigned lengths $\{\ell_n\}$, and all the twists $\{t_n\}$ are equal to $1/2$. The geometry of $X$ when the lengths converge to infinity is very sensitive to the twists. In fact, uncountably many conditions on the lengths and twists need to hold only to guarantee that the covering group $\Gamma$ is of the first kind (see \cite[Theorem C]{Saric10}). When all the twists are $1/2$, the surface is symmetric, and the uncountably many conditions are reduced to a single condition, which allows us to employ hyperbolic geometry estimates. Surprisingly, there are choices of arbitrarily large cuff lengths that make the covering groups of the {\it half-twist flutes} of the first kind, which, again, by symmetry, is equivalent to the flutes being parabolic (see Theorem \ref{thm:parabolicity half-twists} and its proof, and Corollary \ref{cor:KM1}).

	\begin{theorem}\label{thm:parabolic-fast}
		For every non-decreasing sequence $\ell_n$ there is a sequence $\ell_n'\geq\ell_n$ such that the half-twist flute surface $X(\{\ell_n'\},\{1/2\})$ is parabolic.
	\end{theorem}
	
	The proof of Theorem \ref{thm:parabolic-fast} relies on rather precise estimates on the convergence of a nested sequence of geodesics in $\mathbb{H}$ using a characterization in terms of shears from \cite[Theorem C]{Saric10} (see also \cite[Proposition A.1.]{PandazisSaric}).  We show that $\ell_n'$ can, in fact, be taken arbitrarily larger than $\ell_n$  and still satisfy the conclusion of the theorem. We also prove that only an infinite set $S$ of twists needs to be $1/2$ with the rest of twists being $0$ in order to obtain the same result (see Corollary \ref{thm:parabolicity half-twists}). Note that $S$ can be an extremely sparse subset of $\mathbb{N}$ (like $n_k=k!^{k!}$).

	\begin{theorem}
		\label{thm:main}
		Let $\{ \ell_n\}_{n=1}^{\infty}$ be a non-decreasing sequence of positive numbers with $\lim_{n\to\infty}\ell_n=\infty$. Then for every sequence of twists $t_n\in\{0,1/2\}$  such that 
		$\#\{n: t_n=1/2\}=\infty$ there exist non-decreasing sequences $\ell_n'$  and $\ell_n''$ 
		satisfying $\ell_n''\leq \ell_n\leq \ell_{n}'$  such that  the surfaces
		$X (\{\ell_n'\}, \{t_n\})\}$  and $X (\{\ell_n''\}, \{t_n\})\}$ are parabolic.
	\end{theorem}
	

	The construction of the sequences $\{ \ell_n'\}$ and $\{ \ell_n''\}$ is such that it specifically works for surfaces with zero or half twists. There is a delicate balance between the lengths of consecutive cuffs in order to achieve parabolicity.  In fact, there are surfaces $X(\{ \ell_n\},\{ 1/2\} )$ and $X(\{ \ell_n'\},\{ 1/2\} )$
	with $\lim_{n\to\infty}(\ell_n-\ell_n')=0$ such that the first surface is parabolic and the second is not.

	\vskip .2 cm
	
	\subsection{Non-planar surfaces}	Flute surfaces are planar, but the above results hold for (non-planar) surfaces with (infinite) genus. Indeed, let $X$ be an infinite genus surface with a single topological end, called the {\it infinite Loch-Ness monster surface}. Let $\alpha_n$ be the disjoint cuffs converging to the end such that the part of $X$ between $\alpha_n$ and $\alpha_{n+1}$ is homeomorphic to a {torus (or a  genus $g$ surface)} minus two disks. Let $\beta_n$ be the geodesic which cuts off the genus between $\alpha_n$ and $\alpha_{n+1}$. Even more generally, we will allow that $\beta_n$ cuts off a higher genus but finite surface between $\alpha_n$ and $\alpha_{n+1}$ (see Figure \ref{fig:LNM} and Corollary \ref{cor:lochness}).
	
	\begin{theorem}
		\label{thm:LochNess}
		Let $X$ be an infinite Loch-Ness monster surface with cuffs $\alpha_n$ converging to the topological end and cuffs $\beta_n$ that cut off the genus between $\alpha_n$ and $\alpha_{n+1}$. We assume that $\ell (\beta_n)\leq M<\infty$ for all $n$. Given a sequence $a_n$ of positive numbers, there exists a choice of lengths $\ell (\alpha_n)\geq a_n$ such that the infinite Loch-Ness monster surface $X$ with half-twists and chosen lengths is parabolic. Moreover, the lengths can be chosen such that $\ell (\alpha_n)=a_n$ except on an infinite subsequence of $\mathbb{N}$ where the lengths could be larger.
	\end{theorem}
	
	\begin{figure}
		\begin{picture}(280,162)
			\put(39,-2){$\alpha_1$}
			\put(47,50){$\beta_1$}
			\centering
			\includegraphics[width=10 cm]{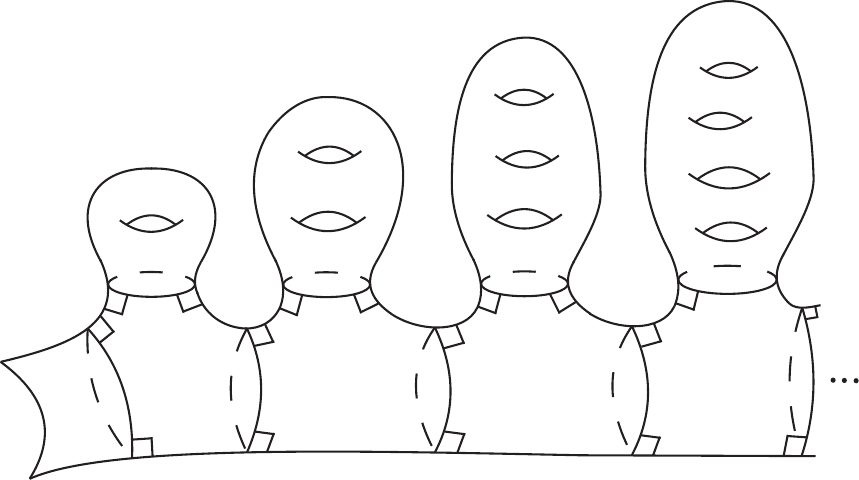}
		\end{picture}
		\caption{The infinite Loch-Ness monster surface}
		\label{fig:LNM}
	\end{figure}

	The above theorem is true for a surface obtained by attaching finitely many infinite Loch-Ness monsters and/or infinite flute surfaces to a finite area bordered surface (see Figure \ref{fig:finite non_planar ends}). In fact, if a surface $X$ has countably many ends $\mathcal{E}$ then we associate a bordered subsurface $X_{e_j}$ for each end $e_j$ that is not a puncture. The subsurface $X_{e_j}$ accumulates only to $e$ and has countably many closed geodesics on its border. Other subsurfaces corresponding to other ends are attached to these border geodesics. Each $X_{e_j}$ is either a flute surface, a Loch-Ness monster surface, or a Loch-Ness monster with truncated genus and some punctures in the place of the genus (see \S \ref{sec:prelim}). Denote by $\alpha_{j,n}$ the cuffs in $X_{e_j}$ that accumulate to $e_j$.

	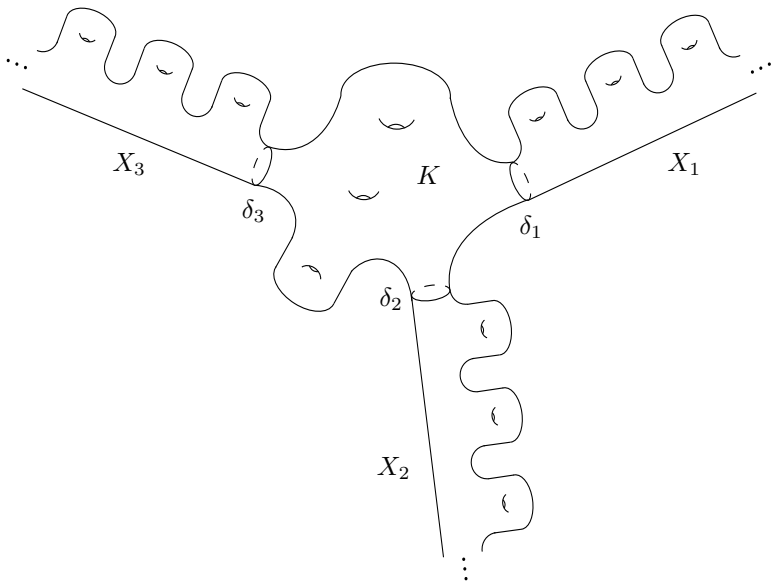
\begin{figure}[htb]
		\begin{picture}(180,230)
			\put(100,150){$K$}
			\put(195,155){$X_1$}
			\put(85,40){$X_2$}
			\put(-15,155){$X_3$}
			\put(139,131){$\delta_1$}
			\put(86,104){$\delta_2$}
			\put(34,137){$\delta_3$}
			\put(-55,0){\includegraphics[width=4in,height=3in,angle=0]{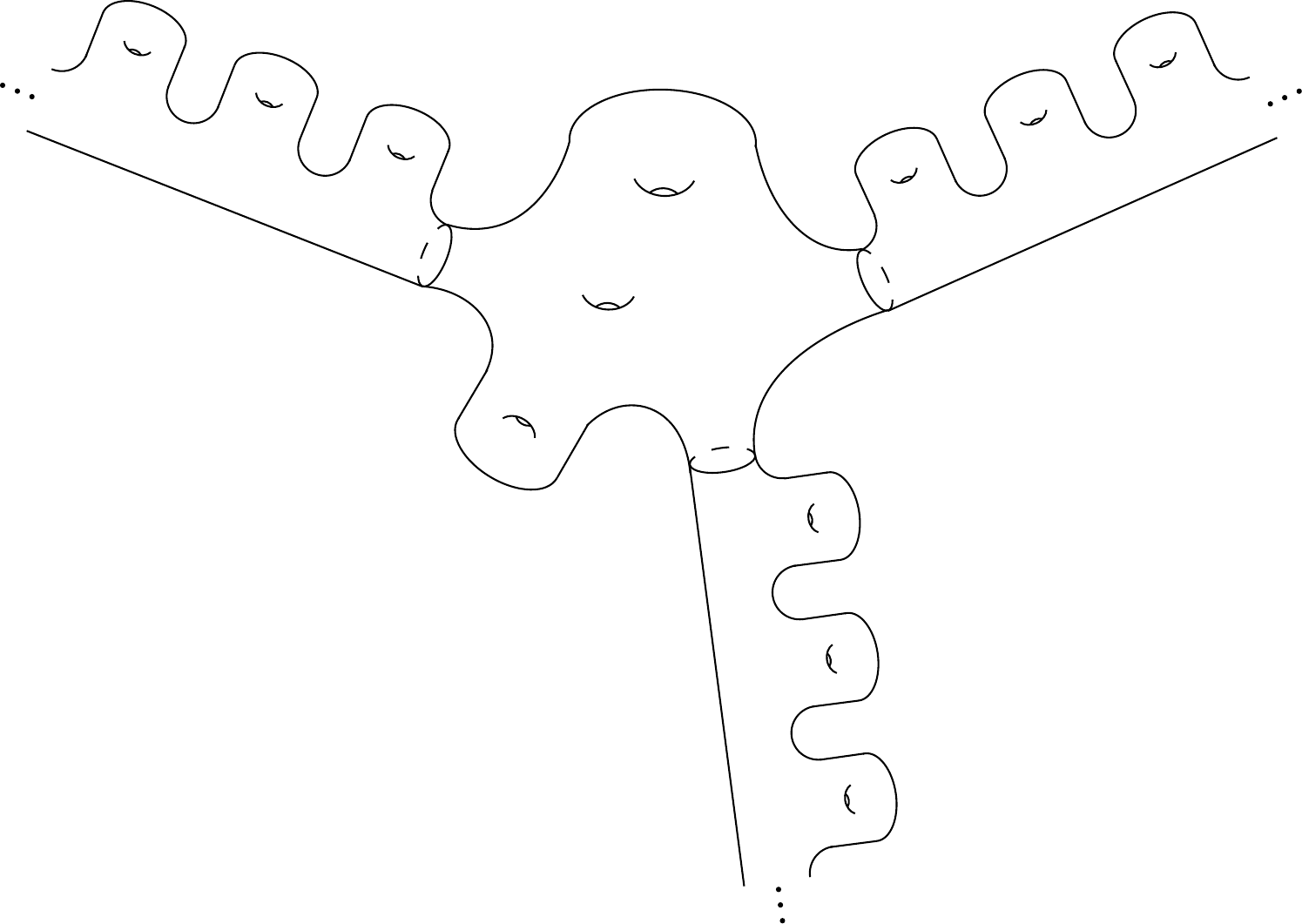}}
		\end{picture}
		\caption{A surface with finitely many non-planar ends.}
		\label{fig:finite non_planar ends}
	\end{figure}

	We prove a reduction theorem
	
	\begin{theorem}
		\label{thm:countable-to-one-end-O_G}
		Consider a Riemann surface $X$ with countably many ends $\mathcal{E}$.
		Then
		$
		X\in O_G
		$
		if and only if, for each $e_j\in\mathcal{E}$, the modulus of the curve family in $X_{e_j}$ that connects $\alpha_{j,1}$ to $e_j$ is zero. 
	\end{theorem}
	
	Using the above reduction theorem, we can extend the examples from single-end surfaces, such as the infinite flute surfaces and the infinite Loch-Ness monster surfaces, to all surfaces with, at most, countably many ends (see Corollary \ref{cor:any-lengths}).
	
	\begin{theorem}
		\label{thm:countable-ends-O_G}
		Let $X$ be a topological surface with countably many ends $\mathcal{E}$. Given a double sequence $\{ a_{j,n}\}$ of positive numbers increasing in $n$ for each fixed $j$, there exists a Riemann surface structure $Y$ on $X$ such that
		$$
		\ell (\alpha_{j,n})\geq a_{j,n}
		$$
		and
		$$
		Y\in O_G.
		$$
	\end{theorem}

	Finally, we remark that surfaces with countably many ends are the largest class of surface for which the above theorem can hold. Indeed, by \cite{Pandazis23} (see also \cite{Saric22}), if $X$ is a Riemann surface with a Cantor set of ends and $\alpha_{j,n}$, $n=1,\ldots , 2^{j+1}$, are cuffs at level $n$ that, for fixed $r>1$, satisfy 
	$$
	\frac{1}{n^2} \gtrsim \ell (\alpha_{j,n})\gtrsim \frac{j^r}{2^j}
	$$
	then
	$
	X\notin O_G.
	$
	\subsection{Overview and open questions}
	One of the main problems we are interested in is to find (hopefully necessary and sufficient) conditions on the Fenchel-Nielsen parameters of a flute surface which would imply parabolicity. The following  table summarizes some ot the recent results from \cite{BHS,PandazisSaric} and the current paper  in the case of symmetric surfaces. 

	\begin{center}
		\scriptsize
		\begin{table}[!h]
		\begin{tabular}{|c|c|c|c|c|c|c|c|l|}
			\hline
			\, & \, & \, & \,  \\
			{\, \bf Twists}\, &  \bf Lengths & \bf Conclusion & \bf Reference\, \\
			\, & \, & \, & \,  \\
			\hline	
			\, & \, & \, & \,  \\
			{\, \bf $t_n=0$} &  no restrictions & $X\in O_{G} \Longleftrightarrow \sum e^{-\ell_n/2} = \infty$ &  \cite{BHS}   \, \\
			\, & \, & \, & \,  \\
			\hline
			\, & \, & \, & \,  \\
			{\, \bf $t_n=1/2$}  \, & \, no restrictions \, & \, $X\in O_{G} \Longleftrightarrow \sum e^{-\sigma_n/2} = \infty$ \, & \, \cite{PandazisSaric} \,      \\
			\, & \, & \, & \,  \\
			\hline
			
			\, & \, & \, & \,  \\
			{\, \bf $t_n=1/2$}  \, & \, $\ell_n$ is concave \, & \, $X\in O_{G} \Longleftrightarrow \sum e^{-\ell_n/4} = \infty$  \, & \, \cite{BHS} \,  \\
			\, & \, & \, & \,  \\
			
			\hline
			\, & \, & \, & \,  \\
			{\, \bf $t_n\in\{0,1/2\}$}  \, & \, no restrictions \, & \, $X\in O_{G} \Longleftarrow \sum e^{-\tilde{\sigma}_n/2} = \infty$ \, & \, Thm \ref{thm:parabolicity half-twists} \,  \\
			\, & \, & \, & \,  \\
			\hline
		\end{tabular}
		\vskip 0.2cm
		\caption{Conditions for parabolicity of tight flutes.}
		\end{table}
	\end{center}
	Thus, a natural question would be the following.
	\begin{problem}
		Suppose $X=X(\{\ell_i\},\{t_i\})$ is a symmetric tight flute surface, i.e., $t_i\in\{0,1/2\}$. Find necessary and sufficient conditions for $X\in O_{G}$. What if $\ell_n$ is concave?
	\end{problem}
		
	Below are the results related to the Kahn-Markovi\'c conjecture and the effect of twisting (i.e., carefull choice of twists) on the type of the surface $X$.
	

 	
	\begin{table}[!h]
		\begin{center}
		\scriptsize
		\begin{tabular}{|c|c|c|c|c|c|c|c|l|}
			\hline
			\, & \, & \, & \,  \\
			{\, \bf Twists}\, &  \bf Lengths & \bf Conclusion & \bf Reference\, \\
			\, & \, & \, & \,  \\
			\hline
		\, & \, & \, & \,  \\
		{\, \bf $t_n'{\in} \, ??$}  \, & \,$ \forall \{\ell_n\}$ \, & \, $X\in O_{G} \overset{??}{\Longleftarrow}$  $(\ell_n,t_n) \rightsquigarrow (\ell_n,t_n')$\, & \, Conj. \ref{conj} \,  \\
		\, & \, & \, & \,  \\
		\hline
		\, & \, & \, & \,  \\
		{\, \bf $t_n'{\in} \{0,1/2\}$}  \, & \,$  \ell_{n_k}=\ell_{n_k+1}$ \, & \, $X\in O_{G} \, {\Longleftarrow}$  $(\ell_n,t_n) \rightsquigarrow (\ell_n,t_n')$\, & \, Cor. \ref{cor:KM1} \,  \\
		\, & \, & \, & \,  \\
		\hline
		\, & \, & \, & \,  \\
		{\, \bf $t_n'{\in} \{0,1/2\}$}  \, & \, $\ell_n'\geq \ell_n$ \, & \, $X\in O_{G} {\Longleftarrow}$  $(\ell_n,t_n)\rightsquigarrow(\ell_n',t_n')$  & \,Thm. \ref{thm:parabolic-fast}  \,  \\
		\, & \, & \, & \,  \\
		\hline
		\, & \, & \, & \,  \\
		{\, \bf $t_n'{\in} \{0,1/2\}$}  \, & \, $\ell_n' = \ell_n$  \, & \, $X\in O_{G} \overset{??}{\Longleftarrow}$  $(\ell_n,t_n)\rightsquigarrow (\ell_n,t_n')$\, & \, Conj. \ref{conj1} \,  \\
		\, & \, & \, & \,  \\
		\hline
		\end{tabular}
		\vskip 0.2cm
					\caption{Parabolicity through twisting.}
					\label{aaa}
	\end{center}
	\end{table}
	%

	As the second and third lines in the table above suggest one may suspect that the twists can in fact be taken to be in $\{0,1/2\}$ to guarantee that $X$ is parabolic. This can be thought of as a strong version of the Kahn-Markovi\'c conjecture.

	\begin{conjecture}[Basmajian, Hakobyan, Pandazis, \v{S}ari\'c]\label{conj1}
		For every non-decreasing sequence of lengths $\{ \ell_n\}$, there is a choice of twists $t_n\in\{0,1/2\}$ s.t. $X(\{ \ell_n\},\{t_n\})$ is parabolic.
	\end{conjecture}
	
	Finally, it would be interesting to find conditions on the Fenchel Nielsen parameters which would guarantee that $X$ is complete but is \emph{not parabolic}. That such flute surfaces exist was shown in \cite{Kinjo}, but no explicit necessary or sufficient condition in terms of the length or twist parameters is known.

	\section{Preliminaries}
	\label{sec:prelim}
	
	A Riemann surface $X$ will always be identified with $\mathbb{H}/\Gamma$, where $\mathbb{H}$ is the upper half-plane and $\Gamma$ is a Fuchsian group. The conformal hyperbolic metric (i.e. the metric of constant curvature $-1$) is induced by the hyperbolic metric on $\mathbb{H}$ since $\Gamma$ acts by isometries on $\mathbb{H}$. The Riemann surface $X$ is said to be {\it infinite} if $\Gamma$ is not finitely generated.
	
	A {\it geodesic pair of pants} is a bordered hyperbolic surface whose interior is homeomorphic to a sphere minus three closed disks and whose boundary components are either simple closed geodesics, called {\it cuffs}, or punctures. We will assume that at least one boundary component is a cuff. Two geodesic pairs of pants $P_1$ and $P_2$ with two cuffs $\alpha_1\subset \partial P_1$ and $\alpha_2\subset \partial P_2$ of equal length can be glued by an isometry along the cuffs to form a more complicated hyperbolic surface (that is homeomorphic to $\mathbb{S}^2$ minus four closed disks). The choice of gluings is determined by a real parameter, called the {\it twist}. Namely, {consider the unique} orthogeodesics from $\alpha_1$ to another boundary component of $P_1$ and from $\alpha_2$ to another boundary component of $P_2$. The signed distance between the feet $x_1\in\alpha_1$ and $x_2\in\alpha_2$ of the orthogeodesics along the identified cuffs $\alpha_1\equiv \alpha_2$ divided by the common length $\ell_X(\alpha_1)=\ell_X(\alpha_2)$ is the (relative) twist $t(\alpha_1)\in [-1/2,1/2]$, where the values $-1/2$ and $1/2$ represent the same gluing. Note that if a pair of pants has two cuffs of equal lengths, then they can be glued by an isometry to produce a bordered surface whose interior is homeomorphic to a torus minus a closed disk.
	
	By taking countably many geodesic pairs of pants $\{ P_n\}_n$ and gluing them by isometries along the cuffs of equal lengths, we obtain an infinite Riemann surface $X$. Let $\{\alpha_j\}_j$ be the family of the images of the cuffs in $X$. The hyperbolic metric (and, by extension, the complex structure) of $X$ is uniquely determined by the lengths $\{\ell (\alpha_j)\}_j$ and the twists $\{ t(\alpha_j)\}_j$. In general, the hyperbolic metric on $X$ may be incomplete, and the natural completion is obtained by attaching the hyperbolic funnels to cuffs not identified with other cuffs and half-planes to bi-infinite simple geodesics which are in the completion of the union of the pairs of pants (see Basmajian \cite{Basmajian}). In fact, every topological pants decomposition of a Riemann surface can be straightened to a geodesic pants decomposition of the interior of the convex core of $X$, and the whole surface $X$ is obtained by attaching the funnels and the geodesic half-planes to the boundary components of the convex core (see \cite{AlvarezRodriguez}, \cite{BasmajianSaric}).

	When we need to add either funnels or the half-planes, the Fuchsian group $\Gamma$ is of the second kind, and the convex core of $X$ is a proper subset of $X$. In the case when the convex core is the whole surfaces $X$, then $\Gamma$ is of the first kind, and $X$ is the union of countably many pairs of pants. We are mainly interested in the groups of the first kind. The pair of sequences {$(\{ \ell (\alpha_n)> 0\},\{ t(\alpha_n)\in(-1/2,1/2]\})$} are called the Fenchel-Nielsen parameters and they determine $X$.
	
	A Green's function on a Riemann surface $X$ is a harmonic function $g:X\setminus \{ z_0\}\to\mathbb{R}$ with $g(z)= -\log |z-z_0|+o(1)$ for $z\in X$ near the fixed point $z_0\in X$ and $g(z)\to 0$ as $z$ leaves every compact subset of $X$. A Riemann surface $X$ is said to be {\it parabolic}, in notation $X\in O_G$, if it does not support a Green's function (see \cite{AhlforsSario}). It is known that $X\in O_G$ if and only if the geodesic flow on the unit tangent bundle $T^1(X)$ is ergodic (see \cite{Nicholls}, \cite{Sullivan}).

	A topological end of a Riemann surface $X$ is an equivalence class of nested components {$\{ U_{n_j}\}_{n_j}$} {given by} the complements of {compact sets $\{K_n\}_n$ that form an exhaustion of $X$} (see \cite{Richards}).  {We say such an end is}
	{\it accumulated by genus} {if each $U_{n_j}$ corresponding to the end has positive genus.} The space of ends $\mathcal{E}$ is a closed subset of a Cantor set, and $X\cup\mathcal{E}$ is a compact set containing $X$ as an open, dense subset. An end $e\in\mathcal{E}$ is {\it simple} if it corresponds to a puncture or a funnel, i.e., if it is not accumulated by other ends or a genus. 
	
	We are interested in the case when $\mathcal{E}$ is countable. We consider the isolated points of $\mathcal{E}$. If an isolated end $e_1$ is not simple, then it is accumulated by genus. 
	Then, there exists a sequence of cuffs $\{\alpha_k\}_k$ accumulating to the end $e_1$ that are the boundary components of a sequence of components $\{U_j\}_j$ of the complements of compact exhaustion $\{ K_n\}_n$ of $X$ that determines the end $e_1$. The part of $X$ between any two adjacent cuffs is a finite surface with a genus. The first cuff cuts out an infinite  Loch-Ness monster surface with end $e_1$, denoted by $X_{e_1}$ (see Figure \ref{fig:decomposition}). 
	
	\begin{figure}[htb]
		\begin{picture}(180,230)
			\put(-55,0){\includegraphics[width=4in,height=3in,angle=0]{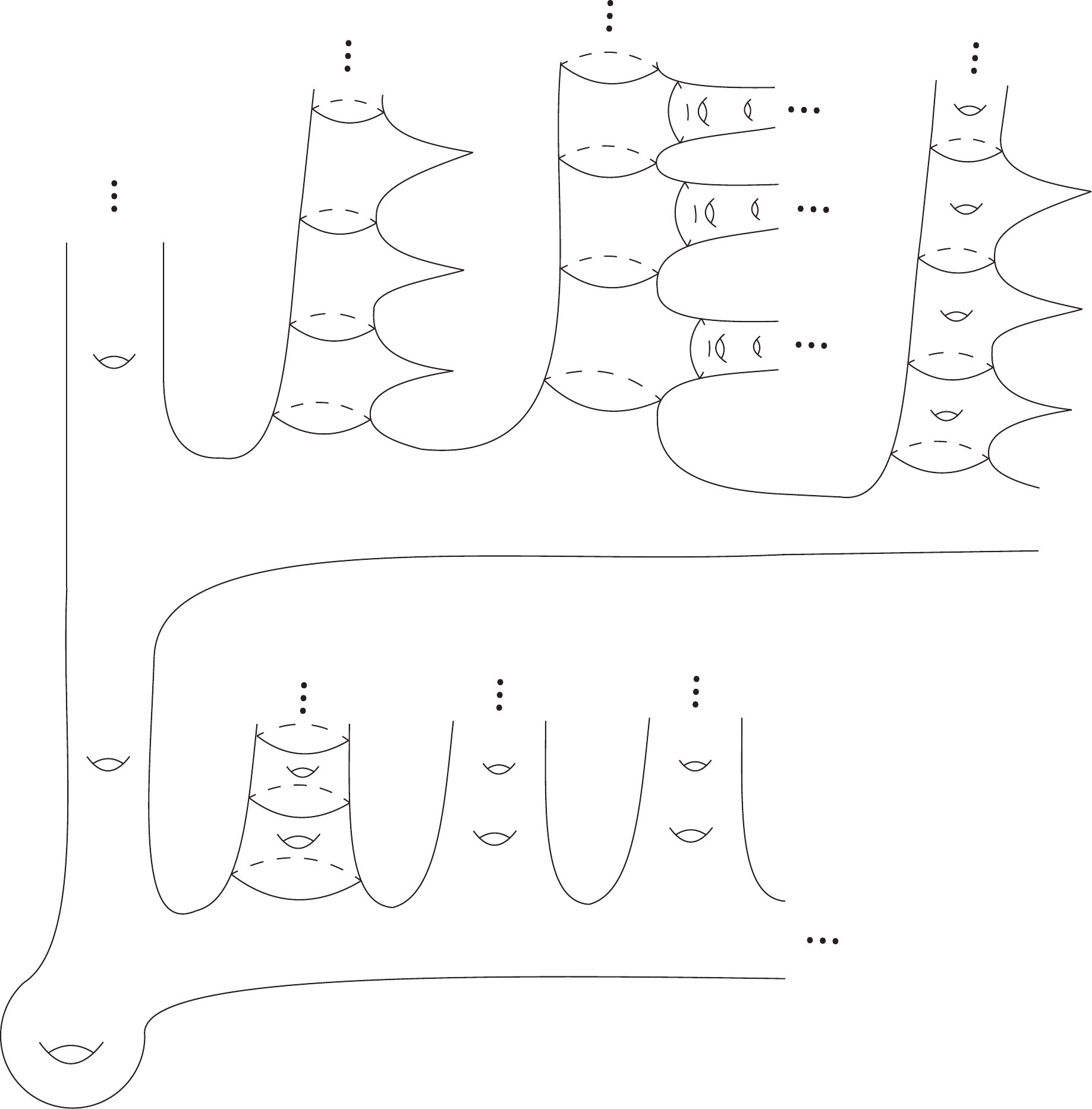}}
			\put(18,33){$\alpha_1$}
			\put(97,128){$\alpha_1$}
			\put(24,123){$\alpha_1$}
			\put(189,116){$\alpha_1$}
			\put(-5,60){$\alpha_2$}
			\put(78,162){$\alpha_2$}
			\put(78,184){$\alpha_3$}
			\put(116,146){$\beta_1$}
			\put(40,60){$X_{e_1}$}
			\put(3,160){$X_{e_2}$}
			\put(229,161){$X_{e_3}$}
			\put(23,90){$e_1$}
			\put(20,210){$e_2$}
			\put(210,205){$e_3$}
			\put(78,210){$e_4$}
			\put(158,136){$e_4^1$}
			\put(158,162){$e_4^2$}
			\put(158,184){$e_4^3$}
		\end{picture}
		\caption{A decomposition into subsurfaces.}
		\label{fig:decomposition}
	\end{figure}

	{Next}, we consider isolated points of $\mathcal{E}'$, where $\mathcal{E}'$ is the set of non-isolated points of $\mathcal{E}$. If an end $e_2\in\mathcal{E}'$ is isolated in $\mathcal{E}'$ and accumulated by simple ends $e_2^k$ in $\mathcal{E}$, then there is a choice of a sequence of cuffs cutting off the components of the complement of a compact exhaustion defining $e_2$. Each cuff of the sequence bounds a tight flute surface with one boundary geodesic (which is the cuff) and one topological end $e_2$. Each two adjacent cuffs bound a geodesic pair of pants with one puncture (see Figure \ref{fig:decomposition}). If an end $e_3\in\mathcal{E}'$ is isolated in $\mathcal{E}'$ and accumulated by both genus and simple ends, then again, we can find a sequence of cuffs that accumulate to the end such that between any two adjacent cuffs, we have a finite surface with genus and punctures. The first cuff cuts out a surface $X_{e_3}$, which is a Loch-Ness monster with punctures (see Figure \ref{fig:decomposition}).
	
	If an isolated point $e_4\in\mathcal{E}'$ is accumulated by a sequence of non-simple ends $e_4^k\in\mathcal{E}\setminus\mathcal{E}'$, then we cut out along the cuffs $\beta_k$ the corresponding (Loch-Ness monster) surfaces of the ends $e_4^k$. We obtained a bordered surface $X_{e_4}$ with one cuff cutting it off from $X$ and with countably many boundary geodesics $\beta_k$. Again, there is a sequence of cuffs $\alpha_n\subset X_{e_4}$ accumulating to the end $e_4$ such that between any two adjacent cuffs, we have a finite subsurface (see Figure \ref{fig:decomposition}). 
	
	If an isolated end $e_5\in (\mathcal{E}')'$ is accumulated by a sequence of ends $e_5^k$ that are isolated in $\mathcal{E}'$, then we can cut off these ends along cuffs $\{\beta_k\}_k$ to obtain the surface $X_{e_5}$  with a sequence of cuffs $\{\alpha_n\}_n$ converging to the end $e_5$ (see Figure \ref{fig:decomposition2}).
	
	\begin{figure}[htb]
		\begin{picture}(180,230)
			\put(-55,0){\includegraphics[width=4in,height=3in,angle=0]{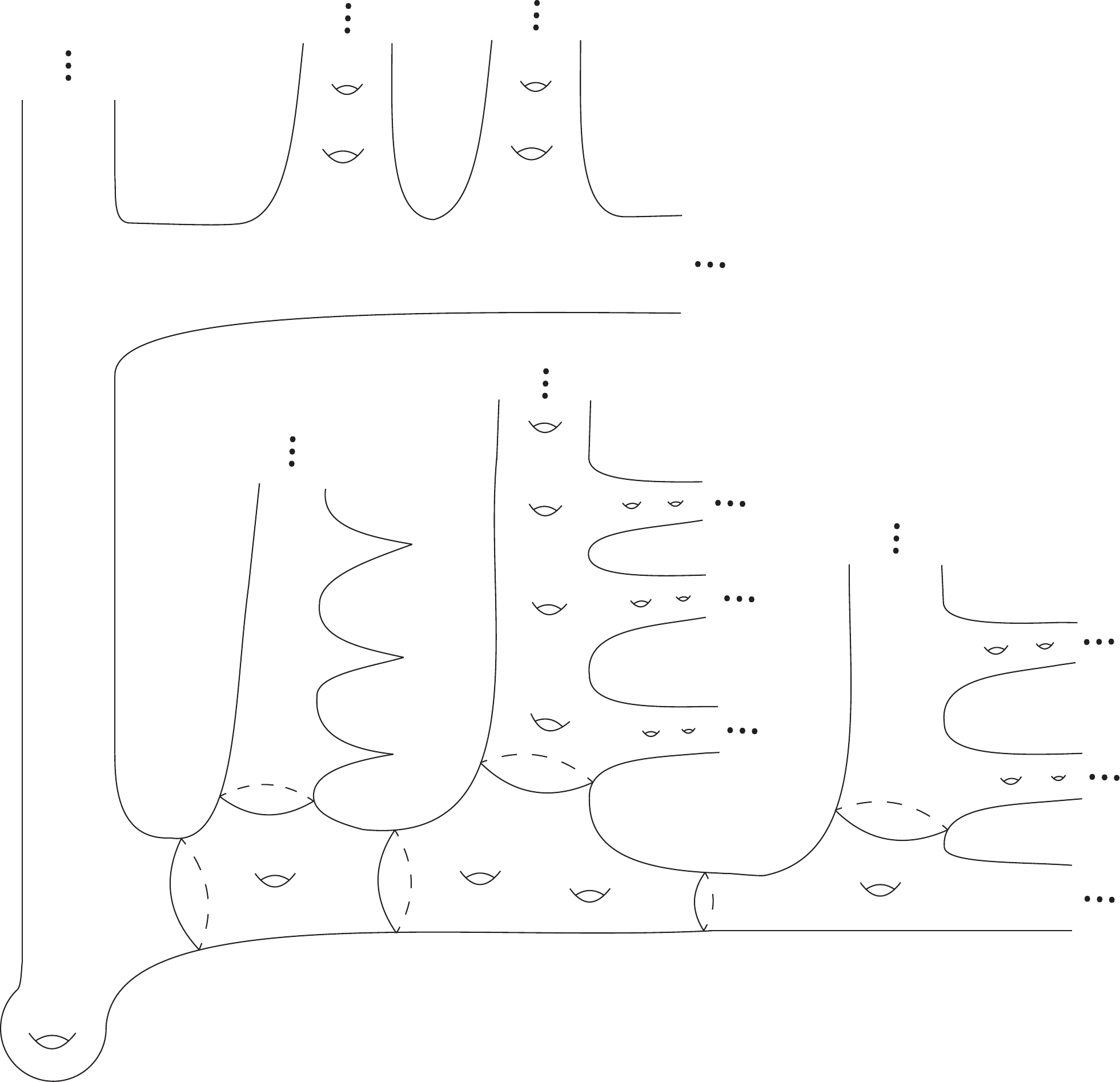}}
			\put(240,34){$e_5$}
			\put(42,93){$X_{e_5^1}$}
			\put(133,104){$X_{e_5^2}$}
			\put(229,72){$X_{e_5^3}$}
			\put(24,125){$e_5^1$}
			\put(90,141){$e_5^2$}
			\put(181,112){$e_5^3$}
			\put(-8,17){$\alpha_1$}
			\put(42,20){$\alpha_2$}
			\put(122,20){$\alpha_3$}
			\put(8,45){$\beta_1$}
			\put(74,49){$\beta_2$}
			\put(178,40){$\beta_3$}
		\end{picture}
		\caption{A decomposition into subsurfaces.}
		\label{fig:decomposition2}
	\end{figure}
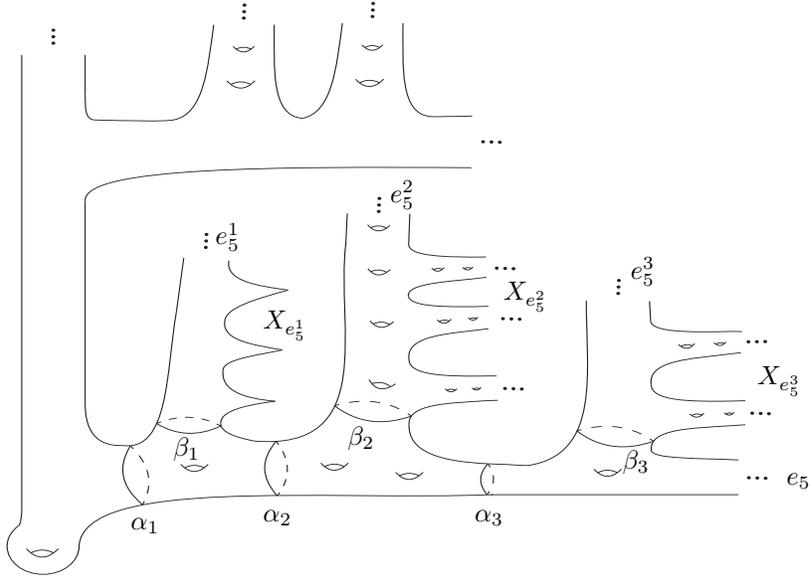

	By continuing this process indefinitely we can 
	make every end to be isolated at a certain stage (over countably many ordinals), see \cite[page 34]{Kechris}. 
	
	Thus to any end $e$, there corresponds 
	a subsurface $X_e$, which accumulates to $e$ whose boundaries are cuffs that are used to cut off subsurfaces from the ends that are isolated before $e$ and that has a sequence of disjoint cuffs accumulating to $e$ as in Figures \ref{fig:decomposition} and \ref{fig:decomposition2}.

	\section{A reduction from countably many to a single end}
	
	When $X$ has countably many topological ends, we find a necessary and sufficient condition for $X\in O_G$ expressed in terms of one end at a time.
	
	Let $\{ X_n\}_{n=1}^{\infty}$ be an exhaustion of $X$ by finite area bordered geodesic subsurfaces. Let $\Gamma_n$ be the curve family in $X_n\setminus X_1$ that connects $\partial X_1$ and $\partial X_n$. 
	Recall that $X\in O_G$ if and only if $\mathrm{mod}(\Gamma_n)\to 0$ as $n\to\infty$ (\cite{AhlforsSario}). 
	
	\begin{definition}
		Let $\Gamma_{\infty}$ be the family of curves in $X\setminus X_1$ that starts at a single point of $\partial X_1$ on one side and accumulate to a unique topological end of $X$ on the other side. 
	\end{definition}
	
	We prove
	
	\begin{theorem}
		\label{thm:mod-ch}
		Let $X$ be any Riemann surface. Then $$X\in O_G\hskip .2 cm \mathrm{ if\ and\ only\ if}
		\hskip .2 cm\mathrm{mod}(\Gamma_{\infty})=0.
		$$
	\end{theorem}
	
	\begin{proof}
		Assume that $X\in O_G$. Then $\mathrm{mod}(\Gamma_n)\to 0$ as $n\to\infty$. Since each curve in $\Gamma_n$ extends to a curve in $\Gamma_{\infty}$, it follows that
		$$
		0\leq \mathrm{mod}(\Gamma_{\infty})\leq \mathrm{mod}(\Gamma_{n}).
		$$
		By letting $n\to\infty$, we conclude that $\mathrm{mod}(\Gamma_{\infty})=0$.
		
		Assume that $\mathrm{mod}(\Gamma_{\infty})=0$. If $X\notin O_G$, then there exists a non-constant harmonic function $u:X
		\setminus \bar{X}_1\to\mathbb{R}$ whose boundary values are $0$ on $\partial X_1$ and that satisfies $0<u<1$ on $X\setminus\bar{X}_1$ (see \cite[page 204, Theorem IV.6C]{AhlforsSario}). The function $u$ is the limit of the solutions $u_n$ to the Dirichlet problem in $X_n\setminus X_1$ with boundary values $0$ on $\partial X_1$ and $1$ on $\partial X_n$. If $u^*$ is a local harmonic conjugate of $u$, then $[d(u+iu^*)]^2$ is a holomorphic quadratic differential on $X\setminus X_1$ whose horizontal trajectories are locally given by $u^*=const$ (see \cite[Theorem 4.2]{Saric22}). The quadratic differential $[d(u+iu^*)]^2$ is integrable since, in the natural parameter, the $du$-length of each $u^*=const$ is at most $1$, and the $du^*$-length of $\partial X_1$ is finite since $\partial X_1$ is compact.

		There is a positive measure set of non-singular horizontal trajectories that start at $\partial X_1$ and converge to the topological boundary of $X$ in the other direction. Indeed, let $e_1$ and $e_2$ be two ends in the accumulation of a horizontal trajectory $r$. The two ends are separated by a finite set of simple closed (hyperbolic) geodesics on $X$, and the trajectory $r$ will cross infinitely many times the (hyperbolic) collar neighborhood of a single closed geodesic. The $|d(u+iu^*)|$-length of each arc of $r$ that connects the two boundary components of the collar is bounded below. Therefore, the length of the trajectory $r$ is infinite. Since there are countably many simple closed geodesics on $X$ and $[d(u+iu^*)]^2$ is integrable, it follows that 
		there can be, at most, a zero-measure set of such trajectories. Since the modulus of the above family of horizontal trajectories is positive (see \cite[Theorem 4.2]{Saric22}), it follows by the monotonicity of the modulus that $\mathrm{mod}(\Gamma_{\infty})>0$. This is a contradiction. Thus $X\in O_G$.
	\end{proof}

	For a Riemann surface $X$ whose space of ends $\mathcal{E}$ is countable, the above theorem can be used to express the parabolicity condition in terms of conditions on $\{ X_e\}_{e\in\mathcal{E}'}$. We prove
	
	\begin{theorem}
		\label{thm:countable-single}
		Let $X$ be a Riemann surface whose space of ends $\mathcal{E}$ is countable. Let $\alpha_e$ be the boundary geodesic of the subsurface $X_e$ (that cuts off $X_e$ from the subsurfaces in the previous generations) corresponding to an end $e\in\mathcal{E}'$ and let $\Gamma_e$ be the family of arcs in $X_e$ connecting $\alpha_e$ to the end $e$. Then $$X\in O_G$$ if and only if
		\begin{equation}
			\label{eq:single_ends}
			\mathrm{mod}(\Gamma_e)=0, \ \mathrm{for\ all\ }e\in\mathcal{E}'.
		\end{equation}
	\end{theorem}
	
	\begin{proof}
		Assume that (\ref{eq:single_ends}) holds. 
		By Theorem \ref{thm:mod-ch}, $X\in O_G$ if and only if $\mathrm{mod}(\Gamma_{\infty})=0$. Recall that $\Gamma_{\infty}$ consists of curves that have one endpoint in a compact subset of $X$ and accumulate to a single topological end in the other direction. This implies that $\Gamma_{\infty}$ overflows $\cup_{e\in\mathcal{E}'}\Gamma_e$. It follows that $\mathrm{mod}(\Gamma_{\infty})\leq \sum_{e\in\mathcal{E}'}\mathrm{mod}(\Gamma_e)=0$. Therefore $X\in O_G$.
		
		Assume that $X\in O_G$. Let $ e\in \mathcal{E}'$. Let $K$ be a finite subsurface of $X$ that contains cuff $\alpha_{e}$ that is on the boundary of $X_{e}$. Since the curve families $\Gamma_{e}\subset X_{e}$ is a subfamily of the curve family $\Gamma_{\infty}$ that connects $K$ with the topological ends, it follows that $\mathrm{mod}(\Gamma_{e})=0$. 
	\end{proof}

	\section{Parabolic flute surfaces with arbitrary large cuffs}
	
	In this section, $X$ is a flute surface with cuffs $\{\alpha_n\}_{n=1}^{\infty}$ as in Figure \ref{fig:flute}. We establish that when the twists around $\alpha_n$ are all $1/2$ then the lengths $\ell_n$ of the cuffs $\alpha_n$ can be chosen to be larger than any prescribed sequence of positive numbers with $X\in O_G$. This answers the original question of Kahn and Markovi\' c of whether one can find flute surfaces with arbitrarily large cuffs.

	First, we establish a sufficient condition for parabolicity for \emph{arbitrary} flute surfaces with zero or half twists, generalizing the result for half-twist surfaces from \cite{PandazisSaric}. 
	We say that a flute surface is a \emph{symmetric flute surface with infinitely many half-twists} if $t_i\in\{0,1/2\}$ for $i\geq 1$, and $\#\{i: t_i=1/2\}=\infty$ (see Figure \ref{fig:half-zero flute}). Note that for such a surface, there always exists an increasing  sequence of positive integers $\{n_k\}$ such that 
	\begin{align}\label{twists}
		t_i=\begin{cases}
			1/2, & \mbox{if } i=n_k,\\
			0, & \mbox{otherwise}.
		\end{cases}
	\end{align}

	\begin{figure}[htb]
		\begin{picture}(180,150)
			\put(-45,90){$t_1=0$}
			\put(0,23){$t_2=0$}
			\put(44,18){$t_3=\frac{1}{2}$}
			\put(70,110){$t_4=\frac{1}{2}$}
			\put(133,116){$t_5=0$}
			\put(203,104){$t_6=\frac{1}{2}$}
			\put(-55,0){\includegraphics[width=10 cm]{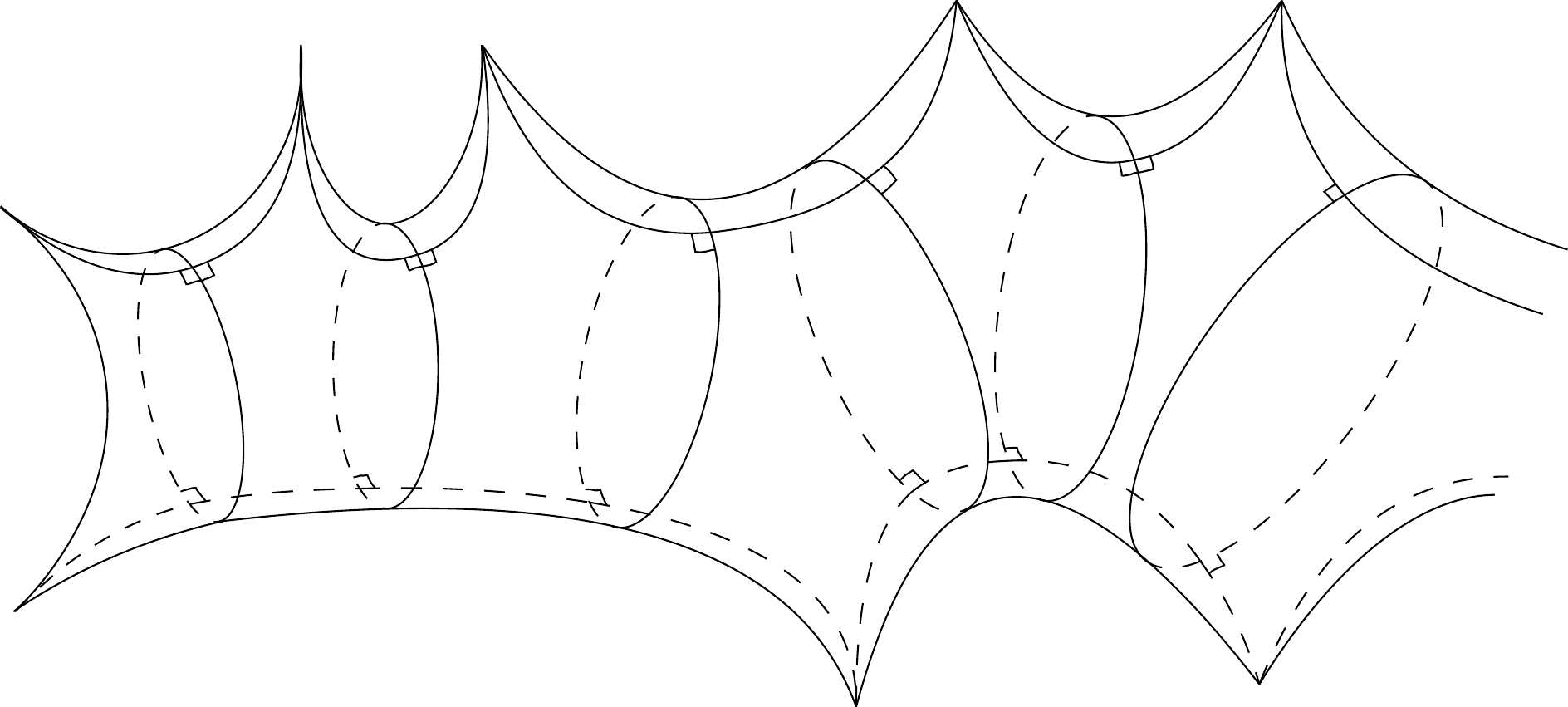}}
		\end{picture}
		\caption{The half and zero twists flute surface.}
		\label{fig:half-zero flute}
	\end{figure}
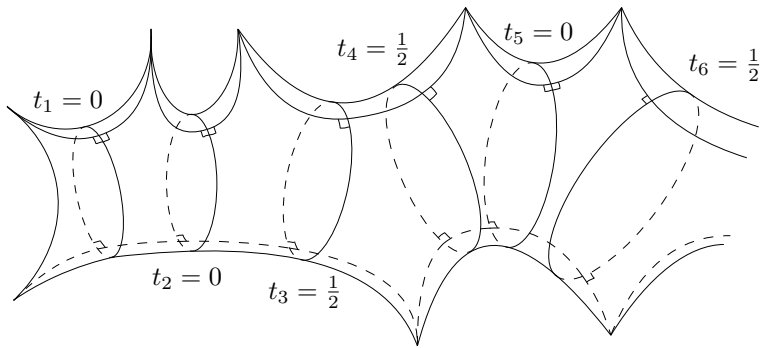

	\begin{theorem}
		\label{thm:parabolicity half-twists}
		Suppose  $X=X(\{\ell_n\}\{t_n\})$
		is a symmetric surface with infinitely many half twists as in \eqref{twists}.
		%
		Then for every non-decreasing sequence $\ell_n$ we have that $X(\{ \ell_n\},\{t_n\})\in O_G$, provided
		\begin{align}\label{parabolic:alternating}
			\sum_{k=1}^{\infty} e^{-\sigma_k/2}=\infty,
		\end{align}
		where $\sigma_{k}=\ell_{n_k}-\ell_{n_{k-1}}+\ldots+(-1)^{k-1} \ell_1$.
	\end{theorem}
	
	\begin{proof}
		Let $\Delta_1$ and $\Delta_2$ be ideal geodesic triangles in the upper half-plane $\mathbb{H}$ with disjoint interiors and the common boundary side geodesic $g$. We choose the orientation of the geodesic $g$ such that $\Delta_1$ is on its left side. 
		Consider the orthogonal projections $p_1$ and $p_2$ of the third vertices of $\Delta_1$ and $\Delta_2$ to the common geodesic $g$. 
		The {\it shear} $s(g)$ on the geodesic $g$ of the pair ($\Delta_1$, $\Delta_2$) is the signed hyperbolic distance (with respect to the orientation of $g$) from $p_1$ to $p_2$  (see \cite{Penner}, \cite{Saric10}). Note that the shear $s(g)$ is independent of the orientation of $g$.

		The symmetric flute surface $X$ can be divided into the front and back sides by geodesics connecting the punctures (see Figure \ref{fig:half-zero flute}). The dividing geodesics partition each pair of pants (except the first one) into two isometric regions: the front and the back pentagons with four right angles and one zero angle. There is an orientation reversing isometry of $X$, which pointwise fixes the dividing geodesics and maps each front pentagon to the corresponding back pentagon. The front side $X^*$ of $X$ is planar, and we fix a single lift of the front side $X^*$ to the universal covering $\mathbb{H}$. The lift $\tilde{X}^*$ is an infinite ideal polygon (see Figure \ref{fig:zero and half twist flute lift}), and the covering map is a conformal map of the polygon onto $X^*$.
		Let $\{g_{2n-1}\}$ be the (infinite) geodesics in $\mathbb{H}$ that are lifts of the cuffs $\alpha_n$ and intersect $\tilde{X}^*$, and $\{g_{2n}\}$ the geodesics which share the initial endpoint with $g_{2n-1}$ and the terminal endpoint with $g_{2n+1}$.

		By the front-to-back symmetry of $X$, it follows that the family of curves that connects a finite area subsurface of $X$ with the non-simple topological end (in the case of flute surface, the non-simple topological end is accumulated by punctures) has zero modulus if and only if the curve family connecting a finite area sub-polygon of $\tilde{X}^*$ to the boundary at infinity not corresponding to punctures has zero modulus. Indeed, the orientation-reversing conformal map of $X$, which maps $X^*$ to the back side of $X$, transfers allowable metrics for $X$ to allowable metrics for $X^*$ with a multiplicative constant $2$. Therefore, one family has zero modulus if and only if the other family has zero modulus. The conformal image $\tilde{X}^*$ of $X^*$ is an ideal polygon in $\mathbb{H}$ with countably many ideal vertices corresponding to the punctures. The closure of 
		$\tilde{X}^*$ in $\mathbb{H}\cup\hat{\mathbb{R}}$ in addition to its vertices contains either a single point in 
		$\hat{\mathbb{R}}=\mathbb{R}\cup\{\infty\}$ (which is the accumulation of the vertices) or an interval. If it is a point, then the modulus of the above family is zero, and $X$ is parabolic. Hence, the covering group is of the first kind. If the closure contains an interval, then $X$ is not parabolic, and the covering group is of the second kind. Thus, the symmetric flute surface is parabolic if and only if its covering group is of the first kind (for more details, see \cite{PandazisSaric}).

		Therefore, the only thing we need to prove is that the infinite polygon $\tilde{X}^{\star}$ accumulates to one point on $\hat{\mathbb{R}}$ in addition to its ideal vertices. The fronts of the geodesic boundaries $\alpha_n$ of $X$ lift to the curves $\tilde{\alpha}_n$ in the universal cover $\mathbb{H}$. These lifts $\tilde{\alpha}_n$ lay on corresponding geodesics we call $g_{2n-1}$ (see Figure \ref{fig:zero and half twist flute lift}). Then define the geodesics $g_{2n}$ to connect the initial point of $g_{2n-1}$ and the terminal point of $g_{2n+1}$. The result is a nested sequence of geodesics $g_n$ such that $g_n$ and $g_{n+1}$ share an endpoint and no three geodesics have a common endpoint. We compute the shears $s(g_n)$ for $n \ge 2$ of the geodesics $g_n$.
		
		\begin{figure}[htb]
			\begin{picture}(100,200)
				\put(-30,15){$g_1$}
				\put(-6,110){$g_2$}
				\put(-14,6){$g_3$}
				\put(35,100){$g_4$}
				\put(-2,0){$g_5$}
				\put(48,70){$g_6$}
				\put(54,-8){$g_7$}
				\put(80,98){$g_8$}
				\put(68,-4){$g_9$}
				\put(99,82){$g_{10}$}
				\put(79,-2){$g_{11}$}
				\put(-55,0){\includegraphics[width = 2.8in, height = 2.8in]{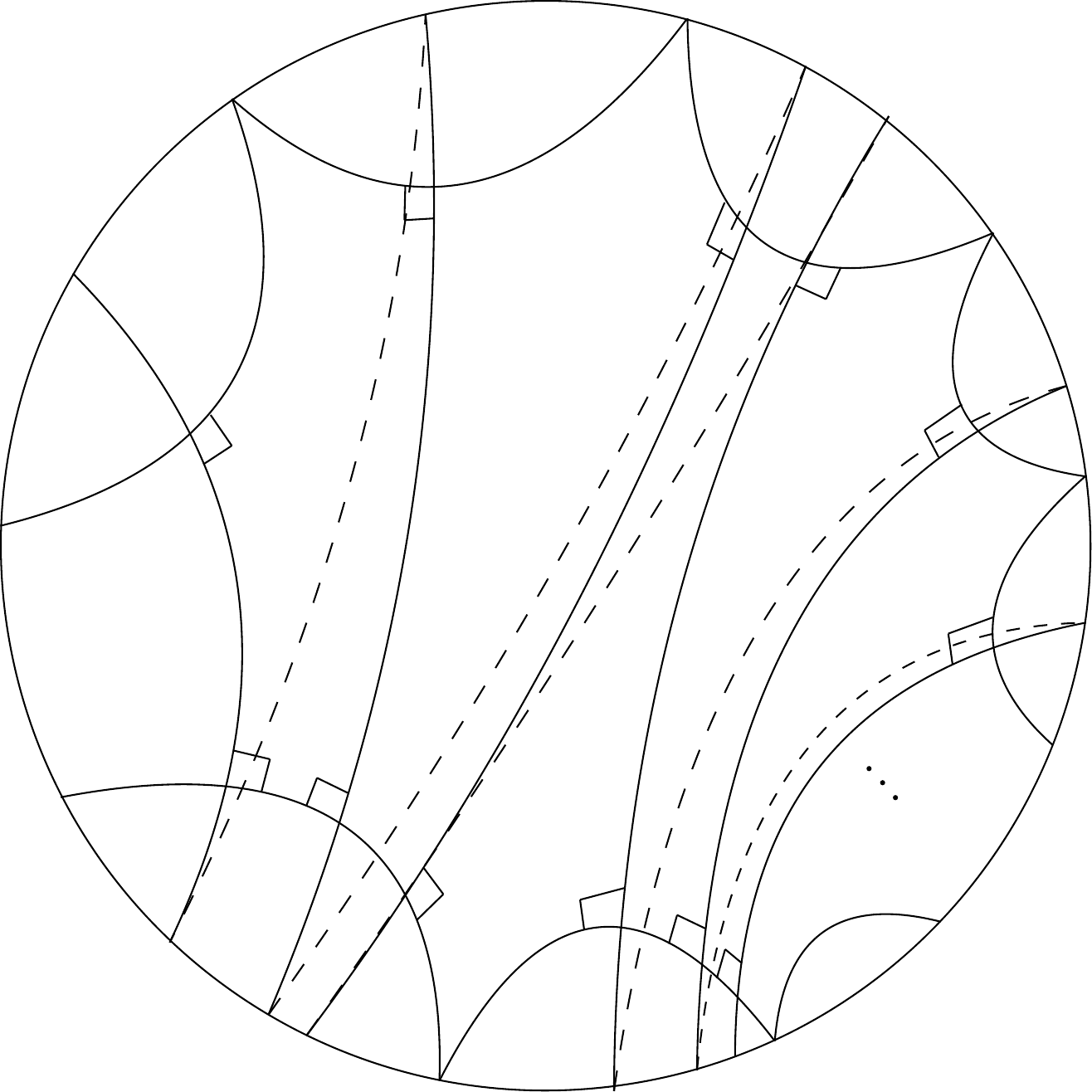}}
			\end{picture}
			\caption{Isometric lift of the front of $X$ to $\mathbb{D}$.}
			\label{fig:zero and half twist flute lift}
		\end{figure}
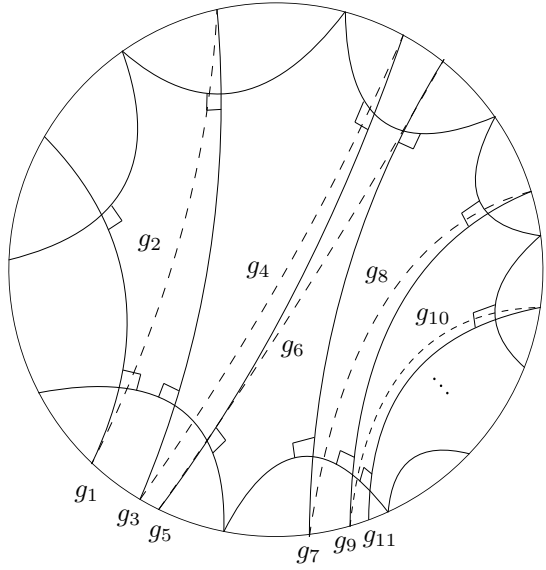
		
		The orthogeodesic arc $\eta_n$ from $g_{2n-1}$ to $g_{2n+1}$ is on the lift of the front side $\tilde{X}^{\star}$ (see Figure \ref{fig:zero and half twist flute lift}). The $\eta_n$, $\tilde{\alpha}_{2n - 1}$, and $\tilde{\alpha}_{2n + 1}$ make up three sides of a geodesic pentagon in $\tilde{X}^{\star}$ that has four right angles and one zero angle. The orthogeodesic ray from the vertex of any of these pentagons with a zero angle to $\eta_n$ separates it into two Lambert quadrilaterals. For each such pentagon, the sides on $g_{2n-1}$ and $g_{2n+1}$ have lengths $\frac{\ell_{n}}{2}$ and $\frac{\ell_{n+1}}{2}$. Using a hyperbolic trigonometry formula for Lambert quadrilaterals \cite[page 38, Theorem 2.3.1(i)]{Buser} gives
		$$
		\ell (\eta_n)=\sinh^{-1}\Big{(}\frac{1}{\sinh \frac{\ell_n}{2}}\Big{)}+\sinh^{-1}\Big{(}\frac{1}{\sinh \frac{\ell_{n+1}}{2}}\Big{)}.
		$$
		It follows that for large $n$,
		\begin{equation}
			\label{eq:eta_n}
			e^{-\frac{\ell_{n+1}}{2}} \lesssim \ell (\eta_n) \lesssim e^{-\frac{\ell_n}{2}}.
		\end{equation}

		We define the {\it cross-ratio} of a quadruple of points $(a,b,c,d)$ in $\hat{\mathbb{R}}$ by
		$$
		cr(a,b,c,d)=\frac{(b-a)(d-c)}{(b-c)(d-a)}.
		$$
		If $g$ is the geodesic with endpoints $(a,c)$, and if $\Delta_1$ is the ideal triangle with vertices $\{ a,c,d\}$ and $\Delta_2$ is the ideal triangle with vertices $\{ a,b,c\}$, then the shear along $e$ with respect to $\Delta_1$ and $\Delta_2$ is (see \cite{SWW}))
		$$
		s(g)=\log cr(a,b,c,d).
		$$
		Use a M\"obius map to map the quadruple of points $(a,b,c,d)$ onto $(-R,-r,r,R)$ for some $0<r<R$. Then the distance $\rho$ between the geodesic $|z|=r$ and $|z|=R$ is given by
		$$
		\rho =\log \frac{R}{r}.
		$$
		On the other hand, we have
		$$
		e^{s(g)}=cr(-R,-r,r,R)=\frac{(R-r)^2}{4rR}.
		$$
		The above gives
		$$
		e^{s(g)}=\sinh^2\frac{\rho}{2}.
		$$
		Since $\eta_n\to 0$ as $n\to\infty$, it follows that
		$s(g_{2n}) < 0$ for large enough $n$ and 
		\begin{equation}
			\label{eq:even shears}
			s(g_{2n})=\log\sinh^2\frac{\ell (\eta_n)}{2}.
		\end{equation}
		
		We orient all geodesics $\{ g_n\}$ to the left as seen from $g_1$. 
		Consider $g_{2n-1}$ that contains the lift of a cuff $\alpha_n$ such that $n_k = n$ for some even $k$ and the corresponding quadrilateral as in Figure \ref{fig:half twist even index shear}. Let $A$ be the initial point of $g_{2n-3}$, and $D$ the terminal point of $g_{2n+1}$. Call $P$ the foot of the orthogedoesic from point $A$ to the geodesic $g_{2n-1}$ and $S$ the foot of the orthogeodesic from point $D$ to $g_{2n-1}$. Let $B\in g_{2n-3}$ and $Q \in g_{2n-1}$ be the endpoints of orthogeodesic $\eta_{n-1}$. Call $R \in g_{2n-1}$ and $C \in g_{2n+1}$ the endpoints of $\eta_{n}$ (see Figure \ref{fig:half twist even index shear}). Note that $\eta_{n-1}$ and $\eta_n$ belong to the boundary sides of the polygon $\tilde{X}^*$, and that the points $(R,P,S,Q)$ appear in that order for the orientation of $g_{2n-1}$. 
		
		The choice of the half-twists and because the index of $g_i$ has remainder $1$ under division by $2$ guarantees that the arc $PS$ is contained in arc $RQ$ (see Figures \ref{fig:zero and half twist flute lift} and \ref{fig:half twist even index shear}). Note that $s(g_{2n-1})$ is the signed distance from $P$ to $S$ for the orientation of $g_{2n-1}$. 
		If $P$ comes before $S$ for the orientation of $g_{2n-1}$, then $s(g_{2n-1})=\ell (PS)$ and 
		$$
		\ell (QR)=\ell (QP)+\ell (RS)-\ell (PS),
		$$
		where $\ell (\cdot )$ is the positive distance. If $P$ comes after $S$, then $s(g_{2n-1})=-\ell (PS)$ and
		$$
		\ell (QR)=\ell (QP)+\ell (RS)+\ell (PS).
		$$
		By the above two equalities and $\ell (QR)=\ell_n/2$, we obtain
		\begin{equation}
			\label{eq: odd shears even index}
			s(g_{2n-1}) = \sinh^{-1}\frac{1}{\sinh \ell(\eta_{n-1})} + \sinh^{-1}\frac{1}{\sinh \ell(\eta_{n})} - \frac{\ell_n}{2}.
		\end{equation}
		$ \newline$
		\begin{figure}[htb]
			\begin{picture}(350,250)
				\put(40,0){\includegraphics[width=3.5in,height=3.5in]{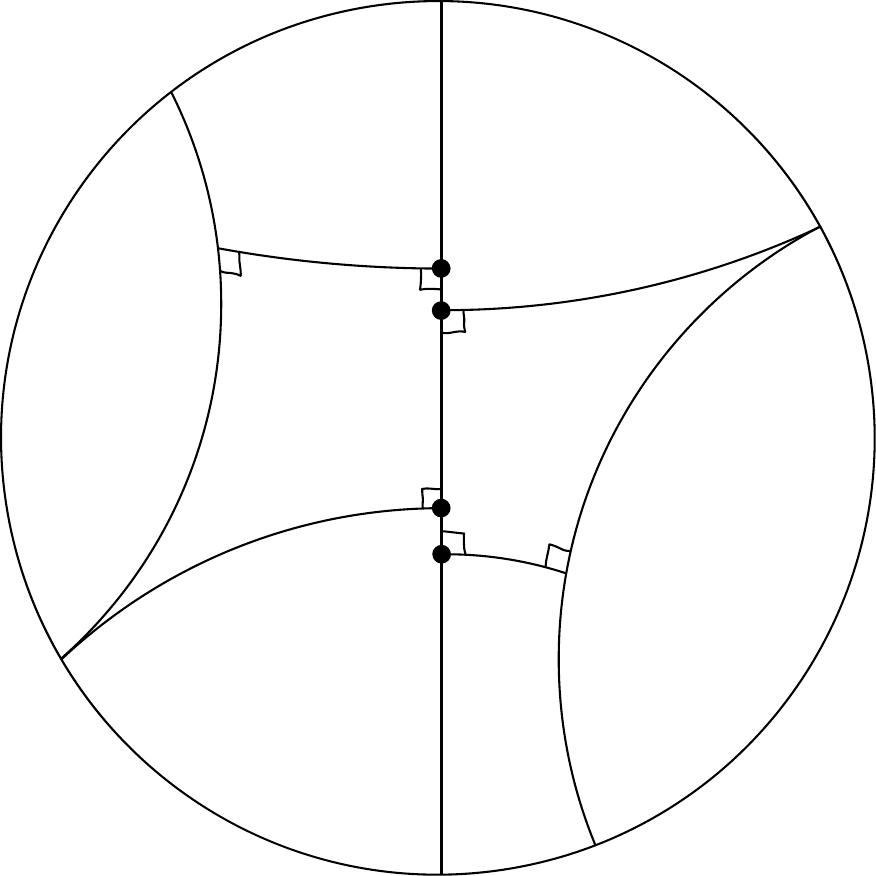}}
				\put(88,178){$B$}
				\put(172,178){$Q$}
				\put(155,156){$S$}
				\put(172,106){$P$}
				\put(152,88){$R$}
				\put(210,80){$C$}
				\put(47,50){$A$}
				\put(282,187){$D$}
				\put(170,20){$g_{2n-1}$}
			\end{picture}
			\caption{$s(g_{2n-1})=\ell (PQ)+\ell (RS)-\frac{\ell_{n}}{2}$ if $t_n = \frac{1}{2}$ and $n_k = n$ for some even $k$.}
			\label{fig:half twist even index shear}
		\end{figure}
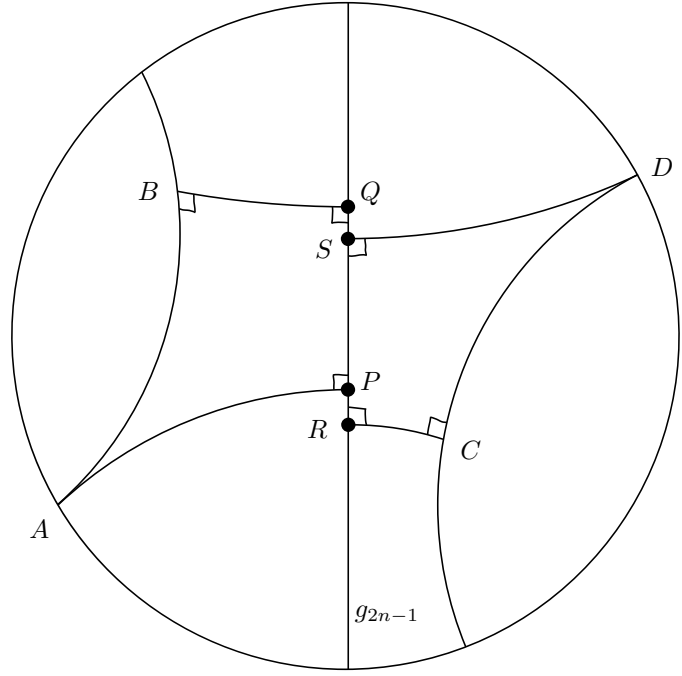
		\indent Consider $g_{2n-1}$ that contains the lift of a cuff $\alpha_n$ such that $n_k = n$ for some odd $k$ and the corresponding quadrilateral in Figure \ref{fig:half twist odd index shear}. Call the vertices of the quadrilateral not on geodesic $g_{2n-1}$, points $A$ and $D$ to the left and right, respectively. Let $P$ be the foot of the orthogeodesic from point $A$ to geodesic $g_{2n-1}$, and let $S$ be the foot of the orthogeodesic from point $D$ to the same geodesic. Call $B$ and $Q$ the endpoints of the orthogeodesic $\eta_{n-1}$ from $g_{2n-3}$ to $g_{2n-1}$. Then call $R$ and $C$ the endpoints of the orthogeodesic $\eta_{n}$ from $g_{2n-1}$ to $g_{2n+1}$ (see Figure \ref{fig:half twist odd index shear}).
		
		The choice of the half-twists and because the index of $g_i$ has remainder $1$ under division by $2$ guarantees that the arc $QR$ is contained in arc $PS$ (see Figures \ref{fig:zero and half twist flute lift} and \ref{fig:half twist odd index shear}). From Figure \ref{fig:half twist odd index shear}, analogous to the case above, we obtain
		\begin{equation}
			\label{eq: odd shears odd index}
			s(g_{2n-1}) = \sinh^{-1}\frac{1}{\sinh \ell(\eta_{n-1})} + \sinh^{-1}\frac{1}{\sinh \ell(\eta_{n})} + \frac{\ell_n}{2}.
		\end{equation}
		
		Consider $g_{2n-1}$ that contains the lift of a cuff $\alpha_n$ such that $t_n = 0$ and the corresponding quadrilateral in Figure \ref{fig:zero twist shear}. Call the vertices of the quadrilateral not on geodesic $g_{2n-1}$, points $A$ and $D$ to the left and right, respectively. Let $P$ be the foot of the orthogeodesic from point $A$ to geodesic $g_{2n-1}$, and let $R$ be the foot of the orthogeodesic from point $D$ to the same geodesic. Call $B$ and $Q$ the endpoints of the orthogeodesic $\eta_{n-1}$ from $g_{2n-3}$ to $g_{2n-1}$. Then call $C$ the endpoint on $g_{2n+1}$ of the orthogeodesic $\eta_{n}$ from $g_{2n-1}$ to $g_{2n+1}$ (see Figure \ref{fig:zero twist shear}).
		$ \newline$
		\begin{figure}[htb]
			\begin{picture}(320,250)
				\put(40,0)
				{\includegraphics[width = 3.5in, height = 3.5in]{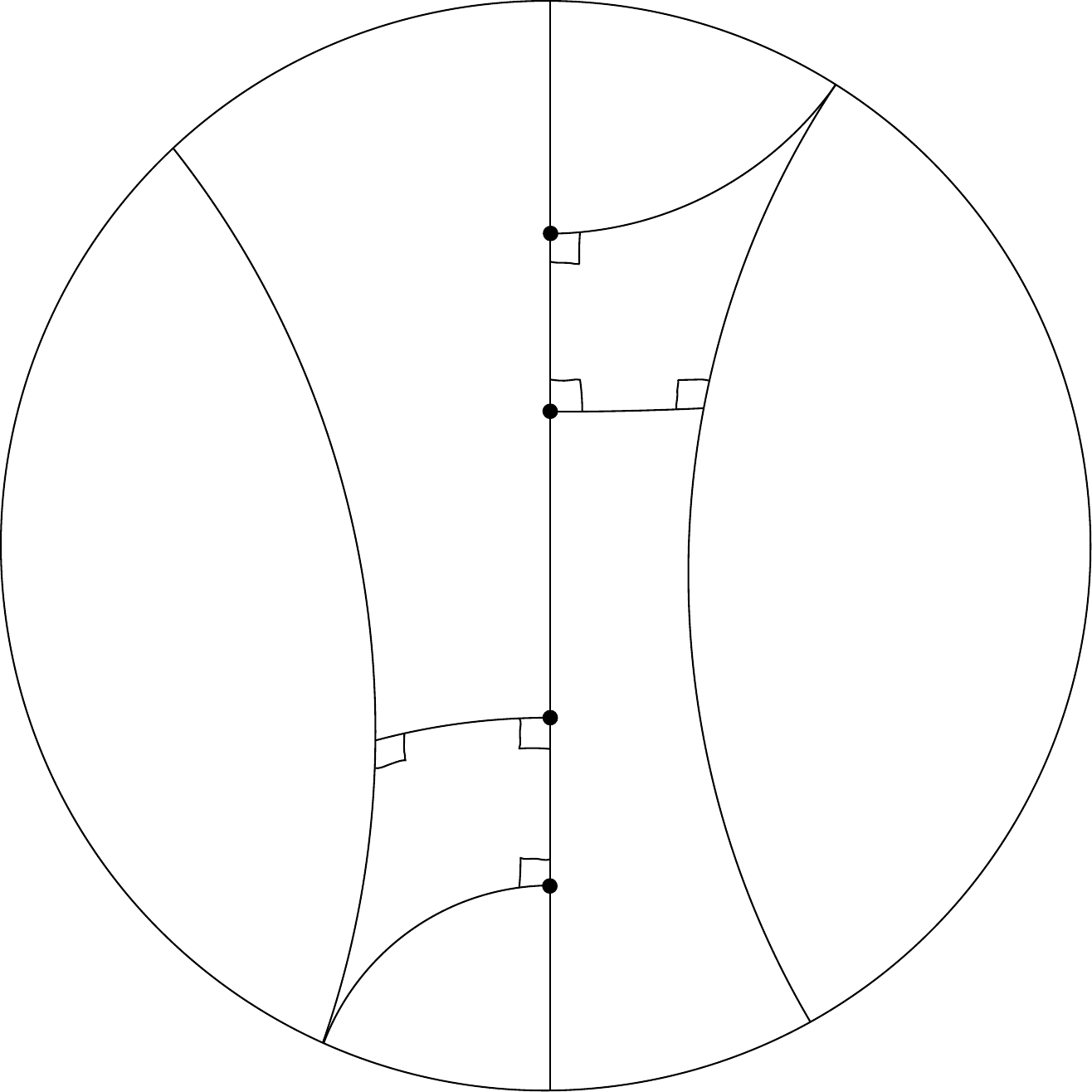}}
				\put(157,195){$S$}
				\put(155,155){$R$}
				\put(206,150){$C$}
				\put(237,235){$D$}
				\put(172,82){$Q$}
				\put(172,35){$P$}
				\put(114,78){$B$}
				\put(103,0){$A$}
				\put(90,125){$g_{2n-3}$}
				\put(172,25){$g_{2n-1}$}
				\put(206,100){$g_{2(n+1)-1}$}
			\end{picture}
			\caption{$s(g_{2n-1}) = \sinh^{-1}\frac{1}{\sinh \ell(\eta_{n-1})} + \sinh^{-1}\frac{1}{\sinh \ell(\eta_{n})} + \frac{\ell_n}{2}$ if $t_n = \frac{1}{2}$ and $n_k = n$ for some odd $k$.}
			\label{fig:half twist odd index shear}
		\end{figure}
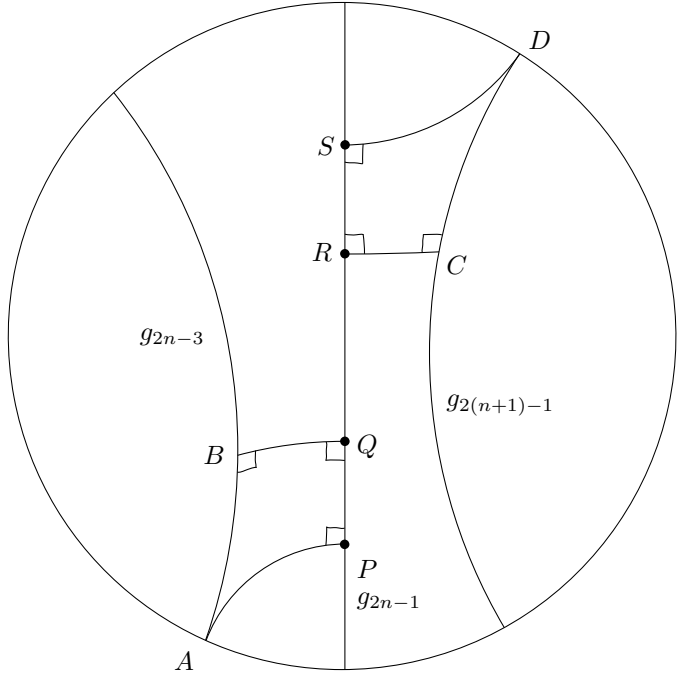
		
		From Figure \ref{fig:zero twist shear} we obtain
		\begin{equation}
			\label{eq: odd shears zero twist}
			s(g_{2n-1}) = \sinh^{-1}\frac{1}{\sinh \ell(\eta_{n-1})} + \sinh^{-1}\frac{1}{\sinh \ell(\eta_{n})}.
		\end{equation}
		
		It will be enough to prove that the sequence of nested geodesics $\{g_n\}_{n=1}^{\infty}$ (see Figure \ref{fig:zero and half twist flute lift}) does not accumulate in $\mathbb{H}$. It is immediate that $\sum_{n=1}^{\infty} \ell(\eta_n) = \infty$ implies that $\tilde{X}^{\star}$ has only one point of accumulation on $\hat{\mathbb{R}}$ in addition to its vertices. Therefore we assume that $\sum_{n=1}^{\infty} \ell(\eta_n) < \infty$ in the rest of the proof. This implies that
		\begin{equation}
			\label{eq:prod-finite}
			1 \le \prod_{n=1}^{\infty} (1 + \ell(\eta_n)) < \infty.
		\end{equation}
		
		By \cite{Saric10}, the sequence $\{g_n\}_{n=1}^{\infty}$ does not accumulate in $\mathbb{H}$ if and only if the piecewise horocyclic path obtained by concatenating the horocyclic arcs between the adjacent geodesics $g_{n-1}$ and $g_n$ has infinite length (see also \cite[Proposition A.1]{PandazisSaric}). Denote
		by $s_n = s(g_n)$ the shear of $g_n$ with respect to the ideal quadrilateral whose vertices are the ideal endpoints of $g_{n-1}$ and $g_{n+1}$ for $n \ge 2$. We do not define
		the shear of $g_1$. We start the piecewise horocyclic path on $g_1$ such that the
		part in the wedge between $g_1$ and $g_2$ has length $1/e^{s_1}$. By \cite{Saric10} and \cite[Proposition A.3]{PandazisSaric},
		the length of the part of the piecewise horocyclic path $h_n$ between $g_n$ and $g_{n+1}$, denoted by $\ell(h_n)$, is
		$ \newline$
		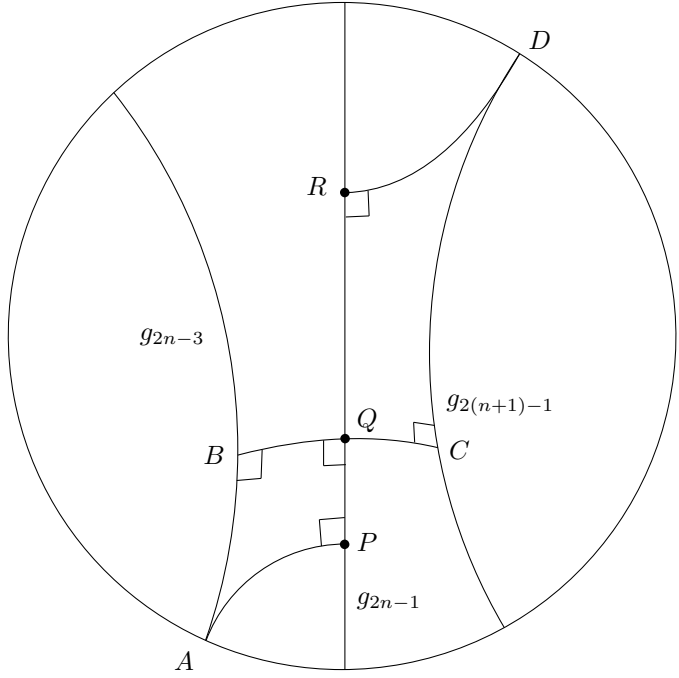
\begin{figure}[htb]
			\begin{picture}(320,250)
				\put(40,0)
				{\includegraphics[width = 3.5in, height = 3.5in]{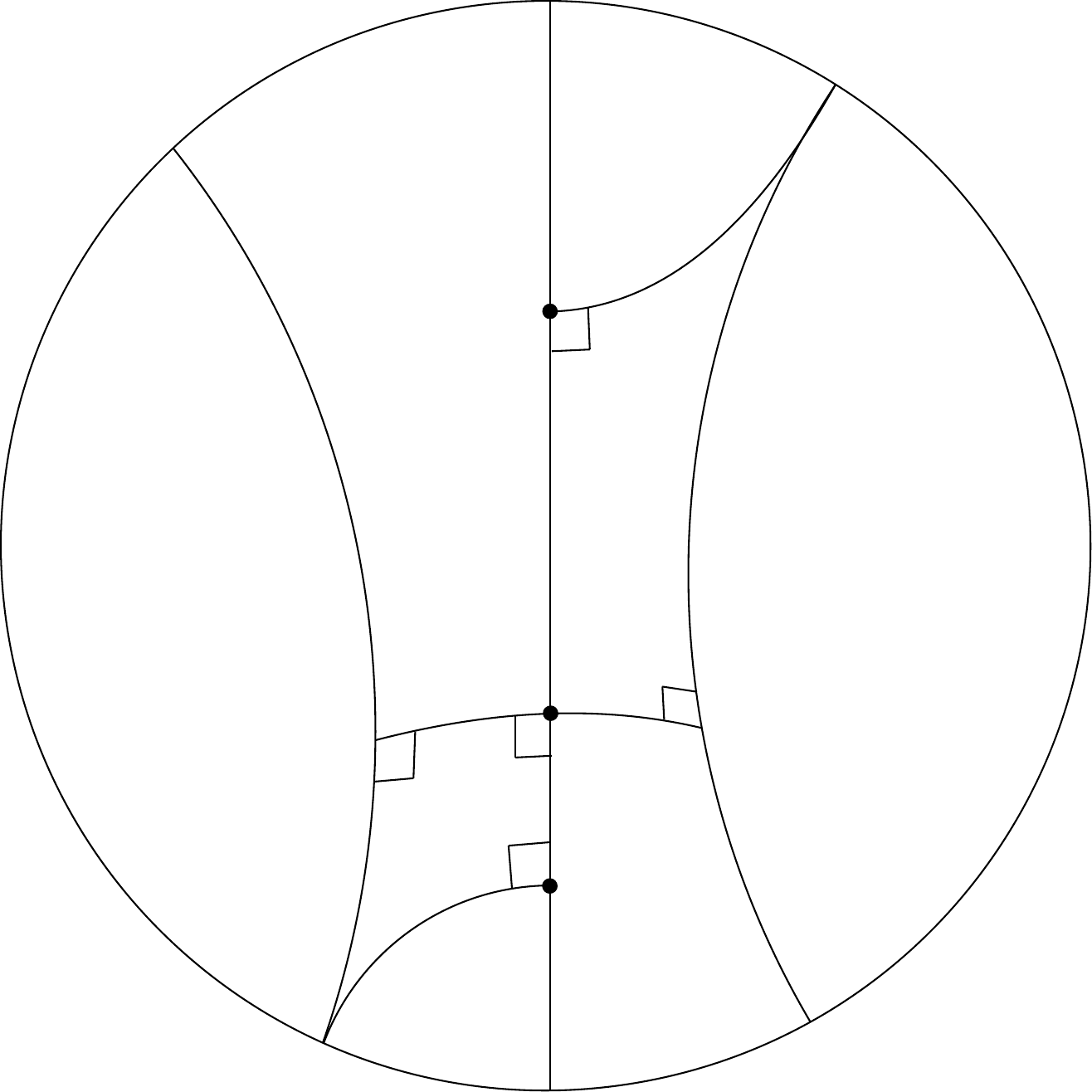}}
				\put(153,180){$R$}
				\put(207,80){$C$}
				\put(237,235){$D$}
				\put(172,92){$Q$}
				\put(172,45){$P$}
				\put(114,78){$B$}
				\put(103,0){$A$}
				\put(90,125){$g_{2n-3}$}
				\put(172,25){$g_{2n-1}$}
				\put(206,100){$g_{2(n+1)-1}$}
			\end{picture}
			\caption{$s(g_{2n-1}) = \sinh^{-1}\frac{1}{\sinh \ell(\eta_{n-1})} + \sinh^{-1}\frac{1}{\sinh \ell(\eta_{n})}$ if $t_n = 0$.}
			\label{fig:zero twist shear}
		\end{figure}
		\begin{align}
			\ell(h_n)=\begin{cases}
				e^{-s_1-s_2-\ldots-s_n}, & \mbox{ if } n \ is \ odd,\\
				e^{s_1+s_2+\ldots+s_n}, & \mbox{if } \ n \ is \ even.
			\end{cases}
		\end{align}
		We will use, for $x>0$, the inequalities $e^{\sinh^{-1}\frac{1}{\sinh x}}>\frac{2}{x}$
		and $\sinh x > x$ to get a lower estimate of the length of the piecewise horocyclic path $\ell (h)$.
		
		We first estimate $\ell (h_{2n})$. By the above, we have
		$$
		\ell (h_{2n})=e^{s_{2n}+\cdots +s_{1}}
		$$
		and we partition the sum of the first $2n$ shears into $n$ consecutive groups $s_{2j}+s_{2j-1}$ for $j=1,\ldots ,n$. For $x>0$, we have the inequalities $e^{\sinh^{-1}\frac{1}{\sinh x}}>\frac{2}{x}$
		and $\sinh x > x$.
		Together with (\ref{eq:even shears})\text{--}(\ref{eq: odd shears zero twist}), this gives 
		$$
		e^{s_{2n}}=\sinh^2\frac{\ell (\eta_n)}{2}>\frac{[\ell (\eta_n)]^2}{4}
		$$
		and
		$$
		e^{s_{2n-1}}>\frac{2}{\ell (\eta_{n-1})}\frac{2}{\ell (\eta_n)}e^{a_n},
		$$
		where 
		\begin{equation*}
			a_n=\left\{
			\begin{array}l
				\ \ \ \ \ \ 0,\ \mathrm{if}\ n\neq n_k\\
				\ \ {\ell_n/2},\ \mathrm{if}\ n=n_k,\ k\ \mathrm{odd}\\
				-{\ell_n/2},\ \mathrm{if}\ n=n_k,\ k\ \mathrm{even}
			\end{array}
			\right.
		\end{equation*}
		
		Then we have
		$$
		e^{s_{2n}+s_{2n-1}}>\frac{\ell (\eta_n)}{\ell (\eta_{n-1})}e^{a_n}
		$$
		which gives
		\begin{equation*}
			e^{s_{2n} + \ldots + s_1}  >\frac{\ell (\eta_n)}{\ell (\eta_1)}e^{a_n+\cdots +a_1}.
		\end{equation*}
		
		By summing over all $n=n_k$ with $k$ odd and using the inequality $\ell(\eta_{n}) > e^{-\frac{\ell_{n}}{2}}$, we get
		\begin{equation}
			\label{eq:sum odd n}
			\sum_{odd \ k} e^{s_{2n_{k}} + ... + s_1}> C\sum_{odd \ k} e^{-\frac{\sigma_{k-1}}{2}}.
		\end{equation}
		
		By (\ref{eq:even shears})\text{--}(\ref{eq: odd shears zero twist}), the inequalities $e^{-\sinh^{-1}\frac{1}{\sinh x}}> \frac{x}{5}$ and $\frac{e^{-\sinh^{-1}\frac{1}{\sinh x}}}{\sinh \frac{x}{2}}>\frac{1}{1+x}$ for small $x>0$, and $\ell(\eta_{n_k}) > e^{-\frac{\ell_{n_k + 1}}{2}} \ge e^{-\frac{\ell_{n_{k+1}}}{2}}$ using an argument similar to the above, we obtain for even $k$
		\begin{equation}
			\label{eq:sum even n}
			\sum_{even \ k} e^{-s_{2n_{k}-1} - ... - s_1}> C\sum_{even \ k} e^{-\frac{\sigma_{k+1}}{2}}.
		\end{equation}
		By summing (\ref{eq:sum odd n}) and (\ref{eq:sum even n}) we obtain {for some constant $C>0$} that the piecewise horocyclic path has length $\ell (h)$ greater than
		$$
		C\sum_{k=1}^{\infty} e^{-\frac{\sigma_{k}}{2}}
		$$
		and the assumption of the theorem implies that it is of infinite length. Thus $\tilde{X}^{*}$ accumulates to exactly one point in addition to its vertices. This implies that $X$ is parabolic.
	\end{proof}

	\begin{figure}[htb]
		\centering
		\includegraphics[width=4.5in,height=2in,angle=0]{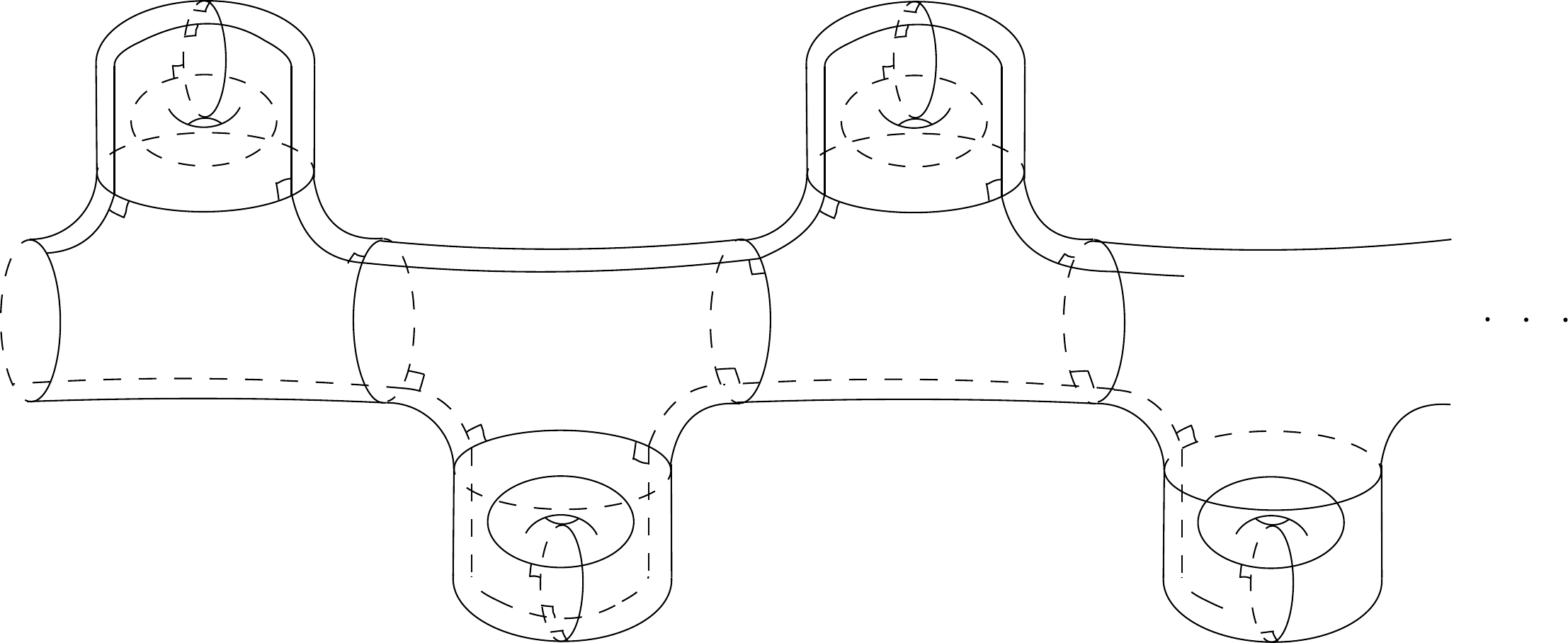}
		\caption{A symmetric end surface with half twists.}
		\label{fig:symm end surface}
	\end{figure}

	Theorem \ref{thm:parabolicity half-twists} implies that Kahn-Markovi\'c conjecture holds for every $\{\ell_n\}$ satisfying \eqref{parabolic:alternating}. In particular, we have the following.
	
	
	\begin{corollary}\label{cor:KM1}
		Suppose $\{\ell_n\}$ is a non-decreasing sequence which has a subsequence  $\{\ell_{n_k}\}$ s.t. for $k\geq1$ we have
		\begin{align}\label{simple-sequence}
			\ell_{n_{2k-1}}=\ell_{n_{2k}}.
		\end{align} 
		Then for the twists $\{t_n\}$ given by (\ref{twists}) the flute surface $X(\{\ell_n\},\{t_n\})$
		is parabolic.
	\end{corollary}
	\begin{proof}
		By \eqref{simple-sequence} we have $\sigma_{2k}=\ell_{n_{2k}}-\ell_{n_{2k-1}}+\ldots+\ell_{n_{2}}-\ell_{n_{1}}=0$,
		and $e^{-\sigma_{2k}} = 1$ for $k\geq 1$. Hence \eqref{parabolic:alternating} holds and $X(\{\ell_n\},\{t_n\})\in O_G$.
	\end{proof}

	\section{Loch-Ness monster surfaces and basic end surfaces}

	In this section, we consider the Loch-Ness monster Riemann surface $X$ with cuffs $\alpha_n$ converging to the infinite end such that the subsurface $Y_n\subset X$ with boundaries $\alpha_n$ and $\alpha_{n+1}$ has genus one. Let $\beta_n$ be the cuff in  $Y_n$ that cuts off the genus of $Y_n$. In other words, $\alpha_n$, $\alpha_{n+1}$, and $\beta_n$ are on the boundary of a pair of pants.
	
	We can form the Loch-Ness monster surface by starting from the bordered surface $X_0$, which has one end at infinity, a sequence of cuffs $\alpha_n$ converging to infinity, and between any two cuffs $\alpha_n$ and $\alpha_{n+1}$ there is a closed boundary geodesic $\beta_n$. Then, to $X_0$, we can attach a genus one surface to each $\beta_n$. The surface $X_0$ can have a puncture instead of $\beta_n$ at arbitrary places. We call this surface $X_0$ the {\it basic end surface}. 
	
	It is also possible to attach a higher genus surface to a border $\beta_n$ of $X_0$ and even an infinite surface. In fact, in Section 1 we associated to each topological end $e$ of $X$ a basic end surface $X_e$.
	
	\begin{theorem}
		\label{thm:basic_end}
		Let $X_0$ be a basic end surface with cuffs $\{\alpha_n\}_n$ converging to the end with increasing lengths $\ell_n$ and border closed geodesics $\{\beta_n\}_n$ with lengths bounded above, where some $\beta_n$ can be punctures. Assume that the twists of $X_0$ around the cuffs $\alpha_n$ with respect to the orthogeodesics from $\beta_n$ are in $\{ 0,1/2\}$ with infinitely many twists $t_{n_k}=1/2$. Then the modulus of the curve family in $X_0$ connecting $\alpha_1$ to the point at infinity is zero under the condition that
		\begin{align}\label{parabolic:alternating}
			\sum_{k=1}^{\infty} e^{-\sigma_k/2}=\infty,
		\end{align}
		where $\sigma_{k}=\ell_{n_k}-\ell_{n_{k-1}}+\ldots+(-1)^{k-1} \ell_1$.
	\end{theorem}
	
	\begin{proof}
		Since the surface $X_0$ is symmetric, it is enough to prove that the curve family in the front half $X^*_0$ has zero modulus. The front half $X^*_0$ is planar and simply connected (since the genus is cut off by $\beta_n$). Therefore $X^*_0$ has a conformal image $\tilde{X}^*_0$ in $\mathbb{H}$, which is an infinite polygon. It is enough to prove that the accumulation of $\tilde{X}^*_0$ on $\hat{\mathbb{R}}$ is a single point. This follows if the lifts of $\alpha_n$ do not accumulate in $\mathbb{H}$. We form the same zig-zag pattern as in the case of the flute surface by drawing extra geodesics between two lifts of the cuffs. The geodesics $\{g_n\}_n$ have asymptotically the same distance $\ell (\eta_n)$ as in the case of the flute surface because $\ell (\beta_n)$ is bounded above. Therefore, the accumulation set of $\tilde{X}^*_0$ is a single point on $\hat{\mathbb{R}}$. The modulus of the curve family is zero.
	\end{proof}
	
	The above theorem has an immediate corollary for the Loch-Ness monster surfaces.
	
	\begin{corollary}
		\label{cor:lochness}
		Let $X_0$ be a basic end surface as above with  $\ell_n=\ell (\alpha_n)$ increasing and closed boundary geodesics $\beta_n$ with $\ell (\beta_n)$ bounded above. Assume that the twists $t_n\in\{ 0,1/2\}$ with infinitely many twists $t_{n_k}=1/2$.
		Let $X$ be a Loch-Ness monster surface obtained by attaching to each $\beta_n$ an arbitrary surface with a finite area (and no boundary). Then
		$$
		X\in O_G
		$$ 
		provided that
		\begin{align}
			\label{parabolic:alternating}
			\sum_{k=1}^{\infty} e^{-\sigma_k/2}=\infty ,
		\end{align}
		where $\sigma_{k}=\ell_{n_k}-\ell_{n_{k-1}}+\ldots+( -1)^{k-1} \ell_1$.
	\end{corollary}
	
	\begin{proof}
		The basic end surface satisfies the condition that the modulus of the curve family going to its end is zero. The attached surface can have a genus going to infinity, but each individual attached surface does not have an infinite end. Therefore, the curve families in the attached surfaces going to infinity have zero modulus as well. By Theorem \ref{thm:countable-single},  the surface $X$ is parabolic.
	\end{proof}

	In fact, the proof of the above theorem only requires that the cuffs $\alpha_n$ lift to geodesics that do not accumulate in $\mathbb{H}$. It is possible that $\ell (\beta_n)$ are unbounded, and this still to be true.
	
	We are also in a position to have a general theorem regarding the parabolicity of surfaces with countably many ends that have arbitrarily large sequences of cuffs going to every topological end.
	
	\begin{theorem}
		\label{thm:main-countable}
		Let $X$ be a Riemann surface with at most countably many ends $\mathcal{E}$. Assume that each subsurface $X_{e_j}$ for $e_j\in\mathcal{E}'$ has a sequence of cuffs $\alpha_{j,n}$ converging to $e_j$ with increasing lengths $\ell_{j,n}=\ell (\alpha_{j,n})$ and twist in $\{ 0,1/2\}$ with infinitely many twists equal to $1/2$. If $\ell_{j,n}$ satisfy (\ref{parabolic:alternating}) for each $j$ then $X$ is parabolic.
	\end{theorem}
	
	\begin{proof}
		By Theorem \ref{thm:countable-single}, it is enough to prove that the family of curves starting on the boundary geodesic of $X_{e_j}$ that is attached to an end surface in the previous level and going to the end $e_j$ has zero modulus. This follows from the condition (\ref{parabolic:alternating}) by Theorem \ref{thm:basic_end}.
	\end{proof}
	
	We note that the condition (\ref{parabolic:alternating}) is satisfied for arbitrarily large cuffs in each end. One example of accomplishing this is to make infinitely many pairs of adjacent cuffs equal to each other while keeping the other lengths as assigned by some double sequence of positive numbers $\{a_{j,n}\}$ increasing in $n$ for each fixed $j$.

	\begin{corollary}
		\label{cor:any-lengths}
		Let $X$ be a topological surface with countably many ends $\mathcal{E}$. Let $X_{e_j}$ be the end surface corresponding to $e_j\in\mathcal{E}'$ and let $\{\alpha_{j,n}\}_{n=1}^{\infty}$ be the cuffs in $X_{e_j}$ accumulating to $e_j$. Given a double sequence $\{ a_{j,n}\}$ of positive numbers increasing in $n$ for each $j$, there exists a Riemann surface structure on $X$ such that $\ell (\alpha_{j,n})\geq a_{j,n}$ and $X\in O_G$.
		
		Moreover, if $n_{j,k}$ is an infinite subsequence and if we choose $\ell (\alpha_{j,n})=a_{j,n}$ for all $n\neq n_{j,k}-1$, and $\ell (\alpha_{j,n_{j,k}-1})=\ell (\alpha_{j,n_{j,k}})$ then $X\in O_G$.
	\end{corollary}

\end{document}